%% file: inverse_main.tex
\documentclass[a4paper]{amsart} 

\newcommand{\polishl}{\l}

\input{preamble}

\usepackage{multirow,booktabs}
\usepackage{scrextend}
\usepackage[outline]{contour}
\usepackage[ruled, vlined]{algorithm2e}

\SetCommentSty{mycommfont}
\usepackage{soul}
\usepackage{sidecap}
\usepackage{tabularx}
\usepackage{amsmath}
 \pgfplotsset{every tick label/.append style={font=\footnotesize},
}

\newcommand{\SCI}{\operatorname{SCI}}

\newcommand{\cO}{\mathcal{O}}

\newcommand{\cA}{\mathfrak{A}}
\newcommand{\cH}{\mathcal{H}}
\newcommand{\cb}{\mathfrak{b}}

\newlength\figurewidth
\newlength\figureheight
\newlength\hfiguremargin
\newlength\vfiguremargin
\newlength{\tablength}

\title{On the complexity of the inverse Sturm-Liouville problem}
\author{Jonathan Ben-Artzi}
\email{Ben-ArtziJ@cardiff.ac.uk}
\author{Marco Marletta}
\email{MarlettaM@cardiff.ac.uk}
\author{Frank R\"{o}sler}
\email{Frank.Roesler@unibe.ch}
\thanks{JBA  acknowledges support from an Engineering and Physical Sciences Research Council Fellowship EP/N020154/1. MM acknowledges support from an Engineering and Physical Sciences Research Council Grant EP/T000902/1. FR acknowledges support from the European Union's Horizon 2020 Research and Innovation Programme under the Marie Sk{\polishl}odowska-Curie grant agreement No 885904.}
\address{School of Mathematics, Cardiff University, Senghennydd Road, Cardiff CF24 4AG, Wales, UK}
\address{Mathematisches Institut,  Universit\"at Bern, Alpeneggstrasse 22,  3012 Bern, Switzerland}
\date\today

\begin{document}
\maketitle
\begin{abstract}
	This paper explores the complexity associated with solving the inverse Sturm-Liouville problem with Robin boundary conditions: given a sequence of eigenvalues and a sequence of norming constants, how many limits does a universal algorithm require to return the potential and boundary conditions?
	It is shown that if all but finitely many of the eigenvalues and norming constants coincide with those for the zero potential then the number of limits is zero, i.e. it is possible to retrieve the potential and boundary conditions \emph{precisely} in finitely many steps. Otherwise, it is shown that this problem requires a single limit; moreover, if one has \emph{a priori} control over how much the eigenvalues and norming constants differ from those of the zero-potential problem, and one knows that the average of the potential is zero, then the computation can be performed with complete error control. This is done in the spirit of the Solvability Complexity Index. All algorithms are provided explicitly along with numerical examples.
\end{abstract}
\setcounter{tocdepth}{1}
\tableofcontents

\section{Introduction and Statement of Main Results}
Inverse problems -- and their reliable computation -- play an important role in many day-to-day applications, such as medical imaging. The purpose of the present article is the rigorous construction of a \emph{one-size-fits-all} algorithm for inverse Sturm-Liouville problems. Namely, we seek an algorithm  that takes as input sequences $\{\lambda_n\}_{n\in\N_0}$ and $\{\alpha_n\}_{n\in\N_0}$\footnote{$\N_0:=\N\cup\{0\}=\{0,1,2,\dots\}$} of eigenvalues and norming constants, respectively, corresponding to the Sturm-Liouville problem (see Section \ref{sec:sturm} below for further discussion)
\begin{align}\label{eq:spectral_problem}
\begin{cases}
	-\psi'' + q(x)\psi = \lambda \psi,&x\in[0,\pi],
	\\
	\hspace{.45cm} \psi'(0) = h\psi(0),
	\\
	\hspace{.4cm} \psi'(\pi) = -H\psi(\pi),
\end{cases}
\end{align}
and returns $q\in L^\infty([0,\pi])$ and $h,H\in\R$. 

The purpose here is not computational efficacy (indeed, this is \emph{not} a paper in numerical analysis) but rather it is to establish whether such an algorithm exists. The framework required for this analysis is furnished by the \emph{Solvability Complexity Index (SCI) Hierarchy}, which is a 
theory for the classification of the computational complexity and limitations of algorithms. This framework has been developed over the last decade by a growing number of authors (cf. \cite{Hansen11,AHS,Ben-Artzi2015a}). This theory is discussed in Section \ref{sec:sci} below, where we define what an `algorithm' is, and in what sense it can `return' $q$, $h$ and $H$. A preliminary version of our main theorem is the following.

\begin{theorem}\label{thm:main1}
Assume that there exist sequences  $\{\kappa_n\}\in\ell^2(\N_0;\C)$, $\{\tilde\kappa_n\}\in\ell^2(\N_0;\R)$ and $\omega\in\R$ such that $\lambda_n$ and $\alpha_n$, $n\in\N_0$, satisfy:
\begin{align}\label{eq:asymptotics}
\begin{aligned}
	\lambda_n^{\f12} &= n + \f{\omega}{\pi(n+1)}+\f{\kappa_n}{n+1},
	\\
	\f{1}{\alpha_n} &= \f{2}{\pi} + \f{\tilde\kappa_n}{n+1}.
\end{aligned}
\end{align}
Then there exists an algorithm which uses only arithmetic operations, for which
	\begin{enumerate}[(i)]
	\item for any $p\in[1,+\infty]$, $q$ can be approximated in $W^{-1,p}([0,\pi])$ and $h,H\in\R$ can also be approximated, though error control for the computation of the triple $(q,h,H)$ is impossible;
	\item if $\omega=0$ and one is given $M>0$ such that $\|\kappa\|_{\ell^2},\|\tilde\kappa\|_{\ell^2}\leq M$, the above approximation can be performed with complete error control;
	\item if only finitely many of the $\kappa_n$ and $\tilde\kappa_n$ are nonzero, $q$, $h$ and $H$ can be computed precisely with finitely many arithmetic operations. 
	\end{enumerate}
\end{theorem}
This theorem is restated in equivalent (and more precise) form in the language of the SCI Hierarchy as Theorem \ref{thm:main2} in the sequel.
\begin{remark}
	We note that the expressions \eqref{eq:asymptotics} are not a numerical requirement, but are necessary for the inverse spectral problem to be well-posed in the first place. In this sense, the above existence result is generic.
\end{remark}
\begin{remark}
The computation of this inverse problem requires us to evaluate trigonometric functions. This evaluation can be included as part of the approximation procedure when this procedure is infinite, but not when the procedure is finite (as in part (iii) of the theorem). In that case we must assume that there exists an \emph{oracle} that can perform such evaluations for us at no additional computational cost (see also Remark \ref{rek:trig} below).
\end{remark}
\subsection{Classical Sturm-Liouville inverse problem}\label{sec:sturm}
The history of the one-dimensional inverse spectral problem for the Sturm-Liouville equation in Liouville normal form goes back to the work 
of Ambarzumjan \cite{Amb}, who proved that only the potential $q\equiv 0$ can give the eigenvalues $\{n^2\}_{n=0}^\infty$
for Neumann boundary conditions. Borg \cite{MR15185} obtained the first general results, for recovery of the potential $q$ from
two spectra (belonging to different boundary conditions), subject to a technical restriction removed by Levinson
\cite{MR32067}. Marchenko's 1950 paper \cite{MR36916}, which generalized these works to prove unique determination
of the potential $q(x)$ by the so-called spectral function $\rho(\lambda)$ and allowed treatment of problems on a semi-axis, marked the 
start of a period of intense research in the Soviet school, culiminating in the Gel'fand-Levitan-Marchenko theory and associated 
integral equation;  overviews of this work may be found in the monographs of Levitan \cite{Levitan}, Marchenko \cite{MR2798059}, and Freiling and Yurko \cite{FreilingYurko}. We also mention the classical text by P\"oschel and Trubowitz \cite{MR894477}. 

Compared to research on numerical algorithms for forward problems, numerical work on inverse problems was sparse. Despite 
the local stability results of Ryabushko \cite{MR734693} and McLaughlin \cite{MR954908}, the inverse problem is well known 
to be ill-conditioned, a fact which is reflected  in the rather weak norm in which Marletta and Weikard \cite{MR2158108} 
estimate errors in $q$ arising from errors in finite spectral data. However by the early 1990s, computing power had reached a 
level which allowed  Andersson \cite{MR1107960} and Rundell and Sacks \cite{Rundell-Saks} to propose algorithms which could run 
on desktop machines. The results in  \cite{Rundell-Saks} show clearly that smoother potentials are recovered more accurately,
something which is explained precisely by the results of Savchuk and Shkalikov \cite{MR2768563}. Nevertheless, the inverse
problem appears to be intrinsically more computationally demanding, especially if one uses an approach which requires the
solution of many `trial' forward problems in order to approximate the solution of the inverse problem. One might 
easily be led by this reasoning to suspect that whatever the complexity of the forward problem (as defined below in Definition \ref{def:sci}), the complexity of the inverse problem
should be greater by at least $+1$. Our main result, Theorem \ref{thm:main1}, shows that this is not true.

Part of the key to obtaining this unexpectedly optimistic result for the inverse problem with full spectral data is the availability
of an algorithm which, for {\em finite spectral data}, recovers a potential $q$ fitting the finite data by solving a finite system of linear 
algebraic equations whose coefficients and right hand side are expressed explicitly in terms of elementary functions of the independent 
variable $x\in [0,\pi]$. The algorithm replaces the (countably infinite) missing spectral data required for a unique solution by the 
values for the free problem $q\equiv 0$; this is equivalent to approximating the data kernel (see \eqref{eq:F_def} 
below) for the Gel'fand-Levitan-Marchenko equation \eqref{eq:integral_equation} by truncating the infinite sum which defines it. In essence, this approach was proposed by McLaughlin and Handelman 
\cite{McLaughlin-Handelman} as a method for creating, from a known Schr\"{o}dinger equation, a new equation with finitely many different 
eigenvalues and norming constants, although no numerical results were presented there.\\

The proof of Theorem \ref{thm:main1} proceeds in the order (iii), (i), (ii). Part (iii) involves a careful analysis of the finite
data algorithm to demonstrate that none of its finitely many steps requires any limiting procedures. Part (i) depends on showing that 
the finite-data potentials converge, in suitable topologies, as the data set expands to a full data set. Part (ii) is the most
technical, as it involves constructing a rigorous set of \emph{a posteriori} error bounds using the data  bound $M$ and the
knowledge that $\omega=0$. Along the way, we prove quantitative versions of the Riesz basis results in Freiling and Yurko 
\cite[\S1.8.5]{FreilingYurko}, which may be of independent interest -- see Proposition \ref{prop:C_Riesz}.\\

In the remainder of this subsection, we set out some notation and basic facts concerning inverse Sturm-Liouville problems which we shall require throughout the rest of our article.

The Sturm-Liouville problem \eqref{eq:spectral_problem} has a sequence $\{\lambda_n\}_{n\in\N_0}$ of eigenvalues and a sequence $\{\alpha_n\}_{n\in\N_0}$ of normalizing constants. The latter are defined as follows. Denote by $\phi(x,\lambda)$ the solution of \eqref{eq:spectral_problem} satisfying $\phi(0,\lambda)=1$. Then for $n\in\N_0$ we define
\begin{align*}
	\alpha_n := \int_0^\pi \phi(x,\lambda_n)^2\,dx.
\end{align*}
By \cite[Thm. 2.10.4-2.10.6]{Levitan} the potential $q$ can be reconstructed from the sequences $\{\lambda_n\}_{n\in\N_0}, \{\alpha_n\}_{n\in\N_0}$. In fact, having a  representation as in \eqref{eq:asymptotics}  is a necessary and sufficient condition in order for there to exist $q,h,H$ such that $\{\lambda_n\}_{n\in\N_0}, \{\alpha_n\}_{n\in\N_0}$ are the spectral data of the problem \eqref{eq:spectral_problem} (cf. \cite[Th. 1.5.2]{FreilingYurko}). In that case the parameter $\omega$ in \eqref{eq:asymptotics}  is  given by
\begin{equation}\label{eq:omega}
	\omega = h+H+\f12\int_0^\pi q(x)\,dx.
\end{equation}

We now follow \cite{FreilingYurko} and summarize the main ideas of the inverse problem. Defining
\begin{align}\label{eq:F_def}
\begin{aligned}
	F(x,y) &= \alpha_0^{-1} \cos(\lambda_0^{\f12}x)\cos(\lambda_0^{\f12}y) - \pi^{-1} 
	\\
	&\qquad + \sum_{n=1}^\infty \big[ \alpha_n^{-1}\cos(\lambda_n^{\f12}x)\cos(\lambda_n^{\f12}y) - 2\pi^{-1} \cos(nx)\cos(ny) \big]
\end{aligned}
\end{align}
and further defining $K(x,y)$ to be the solution of the integral equation
\begin{equation}\label{eq:integral_equation}
	K(x,y) + F(x,y) + \int_0^x K(x,t)F(t,y)\,dt = 0,
\end{equation}
one can retrieve $q$ via the formula 
\begin{align}\label{eq:q_formula}
	q(x) = 2\f{d}{dx} K(x,x).
\end{align}
Moreover, one has
\begin{align*}
	\phi(x,\lambda) = \cos(\lambda^{\f12}x) + \int_0^x K(x,t)\cos(\lambda^{\f12}t)\,dt
\end{align*}
and the boundary conditions can be reconstructed as
\begin{align}\label{eq:BCs}
\begin{aligned}
	h &= K(0,0) = -F(0,0),
	\\
	H &= -\f{\phi'(\pi,\lambda_n)}{\phi(\pi,\lambda_n)},
\end{aligned}
\end{align}
where the expression for $H$ turns out to be independent of $n$ (cf. \cite[Th. 2.10.5]{Levitan}).
It can be shown \cite[Lemma 1.5.4]{FreilingYurko} that if the expressions \eqref{eq:asymptotics} are satisfied then $F$ is continuous on $[0,\pi]^2$. In particular, $F$ is bounded.
\begin{remark}[Finite spectral data]
Observe that if there exists $N\in\N$ such that for all $n> N$ the spectral data is simply $\lambda_n=n^2$ and $\alpha_n=\pi/2$ then the expression \eqref{eq:F_def} for $F(x,y)$ collapses to a finite sum:
\begin{align}\label{eq:F_finite}
\begin{aligned}
	F(x,y) &= \alpha_0^{-1} \cos(\lambda_0^{\f12}x)\cos(\lambda_0^{\f12}y) - \pi^{-1} 	
	\\
	&\qquad + \sum_{n=1}^N \big[ \alpha_n^{-1}\cos(\lambda_n^{\f12}x)\cos(\lambda_n^{\f12}y) - 2\pi^{-1} \cos(nx)\cos(ny) \big].
\end{aligned}
\end{align}
\end{remark}
\subsection{The Solvability Complexity Index Hierarchy}\label{sec:sci}
The Solvability Complexity Index (SCI) Hierarchy addresses questions which are at the nexus of pure and applied mathematics, as well as computer science. Specifically, it provides a classification of the complexity of  problems that can only be computed as the limit of a sequence of approximations. This classification considers how many independent limits are required to solve a problem (for instance, computing the spectrum of elements in $\mathcal{B}(\ell^2(\N))$ requires three independent limits) and whether one can control the approximation errors.
These broad topics are addressed in the sequence of papers \cite{Hansen11,AHS,Ben-Artzi2015a}. 
Research related to this theory has gathered pace in recent years. We point out \cite{Colbrook2019,Colbrook2019a,Colbrook:2019aa} where some of the theory of spectral computations has been further developed; \cite{R19}  where this has been applied to certain classes of unbounded operators; \cite{Becker2020b,colbrook22} where solutions of PDEs were considered; \cite{BMR2022}  where we considered periodic spectral problems; \cite{BMR2020,Ben-Artzi2020a} where we considered resonance problems; and \cite{Colbrook2019c,Webb:2021aa,Colbrook:2021aa,colbrookhorningtownsend} where the authors give further examples of how to perform certain spectral computations with error bounds. Let us summarize the main definitions of the SCI theory.
\begin{de}[Computational problem]\label{def:computational_problem}
	A \emph{computational problem} is a quadruple $(\Om,\Lambda,\Xi,\mathcal M)$, where 
	\begin{enumi}
		\item $\Om$ is a set, called the \emph{primary set},
		\item $\Lambda$ is a set of complex-valued functions on $\Om$, called the \emph{evaluation set},
		\item $\mathcal M$ is a metric space,
		\item $\Xi:\Om\to \mathcal M$ is a map, called the \emph{problem function}.
	\end{enumi}
\end{de}

\begin{de}[General algorithm]\label{def:Algorithm}
	Let $(\Om,\Lambda,\Xi,\mathcal M)$ be a computational problem. A \emph{general algorithm} is a mapping $\Gamma:\Om\to\mathcal M$ such that for each $T\in\Om$ 
	\begin{enumi}
		\item there exists a finite (non-empty) subset $\Lambda_\Gamma(T)\subset\Lambda$,
		\item the action of $\Gamma$ on $T$ depends only on $\{f(T)\}_{f\in\Lambda_\Gamma(T)}$,
		\item for every $S\in\Om$ with $f(T)=f(S)$ for all $f\in\Lambda_\Gamma(T)$ one has $\Lambda_\Gamma(S)=\Lambda_\Gamma(T)$.
	\end{enumi}
\end{de}

\begin{de}[Tower of general algorithms]\label{def:Tower}
	Let $(\Om,\Lambda,\Xi,\mathcal M)$ be a computational problem. A \emph{tower of general algorithms} of height $k$ for $(\Om,\Lambda,\Xi,\mathcal M)$ is a family $\Gamma_{n_k,n_{k-1},\dots,n_1}:\Om\to\mathcal M$ of general algorithms (where $n_i\in\N$ for $1\leq i \leq k$) such that for all $T\in\Om$
	\begin{align*}
		\Xi(T) = \lim_{n_k\to+\infty}\cdots\lim_{n_1\to+\infty}\Gamma_{n_k,\dots,n_1}(T).
	\end{align*}
\end{de}

\begin{de}[Recursiveness]\label{def:recursive}
Suppose that for all $f\in\Lambda$ and for all $T\in\Omega$ we have $f(T)\in \R$ or $\C$. We say that $\Gamma_{n_k,n_{k-1},\dots,n_1}(\{f(T)\}_{f\in\Lambda})$ is \emph{recursive} if it can be executed by a Blum-Shub-Smale (BSS) machine \cite{BSS} that takes $(n_1,n_2,\dots,n_k)$ as input and that has an oracle that can access  $f(T)$ for any $f\in\Lambda$.
\end{de}

\begin{de}[Tower of arithmetic algorithms]\label{def:Arithmetic-Tower}
Given a computational problem $(\Om,\Lambda,\Xi,\mathcal M)$, where $\Lambda$ is countable, a \emph{tower of arithmetic algorithms}  for $(\Om,\Lambda,\Xi,\mathcal M)$ is a general tower of algorithms where the lowest mappings $\Gamma_{n_k,\dots,n_1}:\Omega\to\mathcal{M}$ satisfy the following:
for each $T\in\Omega$ the mapping $\N^k\ni(n_1,\dots,n_k)\mapsto  \Gamma_{n_k,\dots,n_1}(T)=\Gamma_{n_k,\dots,n_1}(\{f(T)\}_{f\in\Lambda(T)})$ is recursive, and $\Gamma_{n_k,\dots,n_1}(T)$ is a finite string of complex numbers that can be identified with an element in $\mathcal{M}$.
\end{de}

\begin{remark}[Types of towers]
One can define many types of towers, see \cite{AHS}. In this paper we write \emph{type $G$} as shorthand for a tower of \emph{general} algorithms, and \emph{type $A$} as shorthand for a tower of \emph{arithmetic} algorithms. If a tower $\{\Gamma_{n_k,n_{k-1},\dots,n_1}\}_{n_i\in\N,\ 1\leq i\leq k}$ is of type $\tau$ (where $\tau\in\{A,G\}$) then we write
	\begin{equation*}
	\{\Gamma_{n_k,n_{k-1},\dots,n_1}\}\in\tau.
	\end{equation*}
\end{remark}

\begin{remark}[Computations over the reals]
The computations in this paper are assumed to take place over the real numbers, hence the appearance of a BSS machine in Definition \ref{def:recursive}. One could attempt to adapt our results to Turing machines -- and this indeed appears to be plausible -- but that is not the purpose of the present paper.
\end{remark}

\begin{de}[SCI]
\label{def:sci}
	A computational problem $(\Om,\Lambda,\Xi,\mathcal M)$ is said to have a \emph{Solvability Complexity Index ($\SCI$)} of $k$ with respect to a tower of algorithms of type $\tau$ if $k$ is the smallest integer for which there exists a tower of algorithms of type $\tau$ of height $k$ for $(\Om,\Lambda,\Xi,\mathcal M)$. We then write \begin{align*}
 			\SCI(\Om,\Lambda,\Xi,\mathcal M)_\tau=k.
		 \end{align*}
If there exist a tower $\{\Gamma_n\}_{n\in\N}\in\tau$ and $N_1\in\N$ such that $\Xi=\Gamma_{N_1}$ then we define $\SCI(\Om,\Lambda,\Xi,\mathcal M)_\tau=0$.
\end{de}

\begin{de}[The SCI Hierarchy]
\label{1st_SCI}
The \emph{$\SCI$ Hierarchy} is a hierarchy $\{\Delta_k^\tau\}_{k\in{\N_0}}$ of classes of computational problems $(\Om,\Lambda,\Xi,\mathcal M)$, where each $\Delta_k^\tau$ is defined as the collection of all computational problems satisfying:
\setlength\tablength{0.15\textwidth}
\begin{tabbing} \hspace{0.8\tablength} \= \hspace{1.7\tablength} \= \hspace{0.6\tablength} \= \hspace{\tablength} \=
\kill
	\> $(\Om,\Lambda,\Xi,\mathcal M)\in\Delta_0^\tau$ \> $\Longleftrightarrow$ \> $\mathrm{SCI}(\Om,\Lambda,\Xi,\mathcal M)_\tau= 0,$ 
	\\[1mm]
	\> $(\Om,\Lambda,\Xi,\mathcal M)\in\Delta_{k+1}^\tau$ \>  $\Longleftrightarrow$ \>  $\mathrm{SCI}(\Om,\Lambda,\Xi,\mathcal M)_\tau\leq k,\qquad k\in\N,$
\end{tabbing}
with the special class $\Delta_1^\tau$  defined as the class of all computational problems in $\Delta_2^\tau$ with known error bounds $\eps_n$:
\begin{tabbing} \hspace{0.8\tablength} \= \hspace{1.7\tablength} \= \hspace{0.6\tablength} \= \hspace{\tablength} \= \kill
	\> $(\Om,\Lambda,\Xi,\mathcal M)\in\Delta_{1}^\tau$ \> $\Longleftrightarrow$ \>
	$\begin{array}{l}
		\exists \{\Gamma_n\}_{n\in \mathbb{N}}\in\tau,\,\exists\epsilon_n\searrow0 
		\\[1mm]
		\text{s.t.} \quad \forall T\in\Omega, \quad d(\Gamma_n(T),\Xi(T)) \leq \epsilon_n.
	\end{array}$
%	 $\exists \{\Gamma_n\}_{n\in \mathbb{N}}\in\tau,\,\exists\epsilon_n\downarrow0$
%	 \\[1mm]
%	 \> \> \> s.t. \quad $\forall T\in\Omega, \quad d(\Gamma_n(T),\Xi(T)) \leq \epsilon_n.$
\end{tabbing}

Hence we have that $\Delta_0^\tau\subset\Delta_1^\tau\subset\Delta_2^\tau\subset\cdots$
\end{de}
\begin{remark}\label{rek:error}
	The definition of $\Delta_1^\tau$ above (using an arbitrary null sequence $\eps_n$) is equivalent to \cite[Def. 6.10]{AHS} where the explicit sequence $2^{-n}$ is used. In fact, given that $d(\Gamma_n(T),\Xi(T)) \leq \epsilon_n$ for some $\eps_n\searrow 0$ one can always achieve $d(\Gamma_{n_k}(T),\Xi(T)) \leq 2^{-k}$ by choosing an appropriate subsequence $n_k$.
\end{remark}
\subsection{Reformulation of Theorem \ref{thm:main1}} 
In view of the setup of Section \ref{sec:sci}, we can now reformulate Theorem \ref{thm:main1} in terms of the language of the SCI Hierarchy. To this end, we first need to define our computational problems.\\

\textbf{Computational problems.}
Fix $M>0$, $N\in\N$ and $p\in[1,+\infty]$. We consider the following \emph{primary sets} whose elements are pairs of sequences of real numbers containing the spectral data:
	\begin{align*}
		\Omega  &:= \left\{(\lambda,\alpha)\in\R^{\N_0}\times\R_+^{\N_0}\,\Big|\,\text{the expressions \eqref{eq:asymptotics} hold}\right\},
		\\
		\Omega_{0,M}  &:= \left\{(\lambda,\alpha)\in\Omega\,\Big|\, \text{in \eqref{eq:asymptotics} }\omega=0 \text{ and }\|\kappa\|_{\ell^2},\|\tilde\kappa\|_{\ell^2}\leq M\right\},
		\\
		\Omega_N  &:= \left\{(\lambda,\alpha)\in\Omega\,\Big|\,\forall n>N,\, \lambda_n=n^2,\, \alpha_n=\pi/2\right\}.
	\end{align*}
These primary sets represent, respectively, arbitrary spectral data, arbitrary spectral data with some known bounds and finite spectral data. We note that the set $\Omega_{0,M}$  includes many interesting operators, such as the case of Neumann boundary conditions ($h=H=0$) with $\int_0^\pi q=0$.\\

The \emph{evaluation set} is, naturally,  the set of individual numbers appearing in the spectral data:\footnote{See also Remark \ref{rek:trig}}
	\begin{align*}
		\Lambda &:= \{ (\lambda,\alpha)\mapsto \lambda_n\,|\, n\in\N_0 \}\cup\{ (\lambda,\alpha)\mapsto \alpha_n\,|\, n\in\N_0 \}.\\
	\end{align*}

The \emph{metric space} should contain the output, which is the potential $q$ and the boundary data $h$ and $H$. We take two different functional spaces for $q$, depending on whether the spectral data is finite (in the case of $\Omega_N$) or infinite (otherwise), hence we define two metric spaces:
	\begin{align*}
	\mathcal{M}_\mathrm{disc}&:=\mathcal{C}([0,\pi])\times\R\times\R,\\
	\mathcal{M}_p&:=W^{-1,p}([0,\pi])\times\R\times\R,
	\end{align*}
where we use the \emph{discrete metric} on $\mathcal{M}_\mathrm{disc}$, that is $d(X,Y)=1$ if $X\neq Y$ and $d(X,Y)=0$ if $X=Y$ for all $X,Y\in \mathcal{M}_\mathrm{disc}$. On $\mathcal{M}_p$, however, we use the canonical metric induced by the natural norms on $W^{-1,p}([0,\pi])$ and $\R$. \\

Finally, the \emph{problem function} is the mapping that returns $q, h$ and $H$. There are two such mappings, corresponding to the two metric spaces:
	\begin{align*}
		\Xi_\mathrm{disc} &:\Omega\to \mathcal{M}_\mathrm{disc},
		\\
		\Xi_p &:\Omega\to \mathcal{M}_p, 
	\end{align*}
and in both cases they map	
	\begin{equation*}
	(\lambda,\alpha)\mapsto(q,h,H).
	\end{equation*}
We shall abuse notation and use the same symbols for the restrictions of these mappings to subspaces of $\Omega$, such as $\Omega_N$ or $\Omega_{0,M}$.\\

\medskip
Armed with these definitions, Theorem \ref{thm:main1} can be reformulated in the following equivalent form.
\begin{theorem}\label{thm:main2}
	For any $N\in\N, M>0$ and $p\in[1,+\infty]$ the computational problems defined above are well-defined and one has
	\begin{align}
		(\Omega,\Lambda,\mathcal{M}_p,\Xi_p) &\in \Delta_2^A 
		\text{ for $p<+\infty$ and }\notin \Delta_1^G \text{ for all }p,	\tag{i}\label{eq:mainth2}
		\\
		(\Omega_{0,M},\Lambda,\mathcal{M}_p,\Xi_p) &\in \Delta_1^A,	\label{eq:mainth3}		\tag{ii}
		\\
		(\Omega_N,\Lambda,\mathcal{M}_\mathrm{disc},\Xi_\mathrm{disc}) &\in \Delta_0^A.		\tag{iii}\label{eq:mainth1}
		\\
\intertext{Moreover, as a direct consequence of  the last result we further have:}
		\big(\cup_{N\in\N}\Omega_N,\Lambda,\mathcal{M}_\mathrm{disc},\Xi_\mathrm{disc}\big) &\in \Delta_2^A. \tag{iv}\label{eq:mainth4}
	\end{align}
	In particular, the computational problem $(\Omega_N,\Lambda,\mathcal{M}_\mathrm{disc},\Xi_\mathrm{disc})$ can be solved exactly with a finite number of arithmetic operations.
\end{theorem} 

\begin{remark}[Evaluating trigonometric functions]\label{rek:trig}
It is well-known that all trigonometric functions, as well as exponentials, can be computed using arithmetic operations to arbitrary precision and with known error bounds. Therefore, for results involving $\Delta_k^A$ with $k\geq1$ we can always incorporate these computations into the tower. However this cannot be done in the case of $\Delta_0^A$ results, since they only involve finitely many arithmetic computations.  In Theorem \ref{thm:main2}\eqref{eq:mainth1} and \eqref{eq:mainth4} (the proof of \eqref{eq:mainth4} follows from \eqref{eq:mainth1}) we must therefore \emph{assume} that there is an oracle which can tell us the values of trigonometric functions at any desired point.
\end{remark}
\begin{remark}[Choice of metric]
	The weak norm used in $\mathcal M_p$ is somewhat natural, given the fact that $q$ is obtained as a derivative (cf. \eqref{eq:q_formula}). While numerical results suggest that convergence might even hold in a strong sense (cf. Section \ref{sec:numerics}), a proof would be highly nontrivial and beyond the scope of this article. A starting point might be to differentiate eq. \eqref{eq:integral_equation_difference} and estimate all newly obtained terms. Under stronger \emph{a priori} assumptions on $q$ it can be shown that convergence in $H^s$ for $s>0$ can be obtained \cite{MR2158108, MR2768563}. Note, however, the different choice of boundary conditions therein (Dirichlet vs. Neumann).
\end{remark}
The result \eqref{eq:mainth4} above follows from the result \eqref{eq:mainth1} quite easily, so we provide the short proof already here. We note that the number $N$ is not needed as input for the algorithm in \eqref{eq:mainth4}.
\begin{proof}[Proof of \eqref{eq:mainth1}$\Rightarrow$\eqref{eq:mainth4}]
	Given $N\in\N$, denote by $\Gamma_N^{\text{fin}}:\Omega_N\to {\mathcal M}_{\text{disc}}$ any $\Delta_0^A$ algorithm that computes $q,h,H$ exactly from $(\lambda,\alpha)\in\Omega_N$. Such an algorithm exists by Theorem \ref{thm:main2}\eqref{eq:mainth1}. Now let $(\lambda,\alpha)\in\bigcup_{N\in\N}\Omega_N$. We define
	\begin{align*}
		\Gamma_N(\lambda,\alpha) := \Gamma_N^{\text{fin}}\Big(\{\lambda_n\}_{n=0}^{N-1}\cup\{n^2\}_{n=N}^\infty,\{\alpha_n\}_{n=0}^{N-1}\cup\{\nicefrac\pi2\}_{n=N}^\infty\Big)
	\end{align*}
	By definition of $\bigcup_N\Omega_N$ there exists $N_0\in\N$ such that $(\lambda,\alpha)\in\Omega_{N_0}$. Therefore for all $N\geq N_0$ we have
	\begin{align*}
		\Gamma_N(\lambda,\alpha) &:= \Gamma_{N}^{\text{fin}}(\lambda,\alpha)
		\\
		&= \Gamma_{N_0}^{\text{fin}}(\lambda,\alpha).
	\end{align*}
 Thus, the sequence $\{\Gamma_N(\lambda,\alpha)\}_{N\in\N}$ is eventually constant and therefore convergent in $\mathcal M_{\mathrm{disc}}$.
\end{proof}
It remains to prove Theorem \ref{thm:main2}\eqref{eq:mainth2}-\eqref{eq:mainth1}. This is done in the sequel.

\section{Proof of Theorem \ref{thm:main2}\eqref{eq:mainth1}: finite spectral data}
\label{sec:finite}
In this section we prove Theorem \ref{thm:main2}\eqref{eq:mainth1} dealing with the case of finite spectral data. Following immediately from the finite sum expression \eqref{eq:F_finite} for $F$, for any $x,y\in[0,\pi]$ the values $F(x,y)$ can be computed using a finite number of arithmetic operations. \emph{We refer the reader to Remark \ref{rek:trig} regarding the evaluation of trigonometric functions.}\\

The first step of the proof is to compute $K(x,y)$ by solving the integral equation \eqref{eq:integral_equation}. To this end, we consider $x$ as a fixed parameter and solve \eqref{eq:integral_equation} as an equation in $y$. Let us introduce the following notation.
\begin{notation}\label{notation}
Define
\begin{itemize}
	\item Right-hand side: $f_x(y) := -F(x,y)$,
	\item Solution: $u_x(y) := K(x,y)$, 
	\item Integral kernel: $k(t,s) := -F(s,t)$.
\end{itemize}
\end{notation}
This transforms \eqref{eq:integral_equation} into the more familiar-looking form
\begin{align}\label{eq:integral_equation2}
	u_x(y) - \int_0^x k(y,s)u_x(s)\,ds = f_x(y).
\end{align}
This is a Fredholm integral equation of the second kind whose kernel is of the form
\begin{align}\label{eq:k_degenerate}
	k(t,s) = \sum_{i=0}^{2N+1} A_i(t)B_i(s)
\end{align}
(cf. \eqref{eq:F_finite}). A concrete choice of $A_i, B_i$ that satisfy \eqref{eq:k_degenerate} is given by
\begin{align}\label{eq:AiBi}
\begin{cases}
	A_i(s) = -B_i(s) = \alpha_i^{-\f12}\cos(\lambda_i^{\f12}s) & \text{ for }i=0,\dots,N,
	\\
	A_i(s) = B_i(s) = \big(\f2\pi\big)^{\f12}\cos((i-N)s) & \text{ for } i=N+1,\dots,2N,
	\\
	A_{2N+1}(s) = B_{2N+1}(s) \equiv \pi^{-\f12}. &
\end{cases}
\end{align}
As detailed in \cite[Ch. 2]{Atkinson}, making the ansatz $u_x(y) = f_x(y)+\sum_{i=0}^{2N+1} c_iA_i(y)$ and plugging it into \eqref{eq:integral_equation2} yields the  finite linear system
\begin{align}\label{eq:linear_system}
	c_i - \sum_{j=0}^{2N+1} \langle A_j,B_i\rangle c_j = \langle f_x,B_i\rangle,
\end{align}
where
\begin{equation*}
	\langle A_j,B_i\rangle = \int_0^x A_j(s)B_i(s)\,ds
\quad\text{and}\quad
	\langle f_x, B_i\rangle = \int_0^x f_x(s)B_i(s)\,ds.
\end{equation*}
\begin{remark}\label{remark:AiBi_analytical}
	We note that both integrals can be calculated \emph{analytically} using \eqref{eq:AiBi} and elementary rules for integration. The result of these integrations will always be a polynomial of degree 2 in $x$, $\cos(\lambda_j x)$, $\cos(j x)$, $\sin(\lambda_j x)$, $\sin(j x)$ for $j=0,\dots,N$, whose coefficients can be computed from the $\lambda_n$, $\alpha_n$.
\end{remark}
Next, we apply two classical results to show that \eqref{eq:linear_system} is uniquely solveble. To simplify notation we denote the $(2N+2)\times(2N+2)$ matrix with entries $\langle A_j,B_i\rangle$ by $\cA$, the vector with entries $\langle f_x, B_i\rangle$ by $\cb$ and the identity matrix by $I$. The linear system \eqref{eq:linear_system} becomes
	\begin{equation*}
	(I-\cA)c=\cb.
	\end{equation*}
We note that a similar approach to inverse problems has been used in \cite{McLaughlin-Handelman}, however not in the context of the SCI Hierarchy.
To avoid confusion in the sequel, we introduce the integral operator $\kappa_x:L^2([0,x])\to L^2([0,x])$ defined as
\begin{equation}\label{eq:int-ker}
	(\kappa_x u)(t) := \int_0^x k(t,s)u(s)\,ds.
\end{equation}
\begin{lemma}\label{lemma:system_solvable}
	The matrix $I-\cA$ is invertible, and hence the system \eqref{eq:linear_system} has a unique solution for every $x\in[0,\pi]$.
\end{lemma}
\begin{proof}
By \cite[Th. 2.1.2]{Atkinson} the system \eqref{eq:linear_system} is nonsingular if the operator $\mathrm{Id}-\kappa_x : L^2([0,x])\to L^2([0,x])$ is invertible, and by \cite[Th. 2.3.1]{Levitan} this operator  is indeed invertible for all $x\in[0,\pi]$.
\end{proof}
\begin{lemma}\label{lemma:cj_rational}
	The solutions $(c_0,\dots,c_{2N+1})$ of \eqref{eq:linear_system} are rational functions of degree $4N+6$ in $x$, $\cos(\lambda_j x)$, $\cos(j x)$, $\sin(\lambda_j x)$, $\sin(j x)$ for $j=0,\dots,N$. They can be computed symbolically in finitely many arithmetic operations from $\{\lambda_n\}_{n=0}^N\cup\{\alpha_n\}_{n=0}^N$.
\end{lemma}
\begin{proof}
	By Lemma \ref{lemma:system_solvable} the system \eqref{eq:linear_system} is solvable and its solution is given by $(I-\cA)^{-1}\cb$. By Remark \ref{remark:AiBi_analytical} the entries of $\cA,\cb$ can be calculated explicitly from $\{\lambda_n\}_{n=0}^N\cup\{\alpha_n\}_{n=0}^N$ as polynomials of degree 2 in $x, \cos(\dots), \sin(\dots)$. But the entries of $(I-\cA)^{-1}$ can  be calculated in finitely many arithmetic operations from the entries of $\cA$ by the formula
	\begin{align}\label{eq:inverse_matrix}
		M^{-1} = \f{1}{\det(M)}\operatorname{adj}(M),
	\end{align}
	where $\operatorname{adj}(M)=(-1)^{i+j}[M]_{ij}$ and $[M]_{ij}$ denotes the $ij$-th minor of $M$. The result is a rational function in $x, \cos(\dots), \sin(\dots)$ of degree $4N+4$.
	Finally, the product $(I-\cA)^{-1}\cb$ is computed in finitely many operations on the entries of $(I-\cA)^{-1}$ and $\cb$ and yields a rational function of degree $4N+6$.
Hence, each $c_i$ has a representation as a rational function in $x$, $\cos(\lambda_j x)$, $\cos(j x)$, $\sin(\lambda_j x)$, $\sin(j x)$ for $j=0,\dots,N$.
\end{proof}
\begin{remark}
	Note that the denominator in the rational function expression for $c_j$ is never zero. Indeed, by \eqref{eq:inverse_matrix} the denominator is given by $\det(I-\cA)$, which is guaranteed to be nonzero for all $x\in[0,\pi]$ by Lemma \ref{lemma:system_solvable}.
\end{remark}
To emphasize the dependence of the $c_j$ on $x$, we will sometimes write $c_j=c_j(x)$ in the following. Having obtained a computable solution $(c_0,\dots,c_{2N+1})$ of \eqref{eq:linear_system}, we recall our ansatz for $u_x$:
\begin{align}\label{eq:u_formula}
	u_x(y) = f_x(y) + \sum_{i=0}^{2N+1} c_i(x)A_i(y).
\end{align}
By construction $u_x$ satisfies \eqref{eq:integral_equation2} for every $x\in[0,\pi]$. By Lemma \ref{lemma:cj_rational} and \eqref{eq:AiBi} the right hand side of \eqref{eq:u_formula} is given symbolically as a rational function in $x$, $\cos(\lambda_j x)$, $\cos(j x)$, $\sin(\lambda_j x)$, $\sin(j x)$ for $j=0,\dots,N$, and likewise for $y$. In particular, the derivative 
\begin{align*}
	q(x) &:= 2\f{d}{dx}u_x(x) 
	\\
	&= 2\f{d}{dx}K(x,x)
\end{align*}
is a rational function again and can be computed symbolically as a function of $x$. Moreover, once $K(x,y) = u_x(y)$ has been computed, the boundary conditions $h,H$ can be reconstructed using \eqref{eq:BCs}. Indeed, $h$ is given by
\begin{align*}
	h = K(0,0),
\end{align*}
and we claim that $H$ is given by
\begin{align*}
	H = -K(\pi,\pi).
\end{align*}
To compute $H$, recall that for $n>N$ one has
\begin{align*}
	H = -\f{\phi'(\pi,n^2)}{\phi(\pi,n^2)}.
\end{align*}
These can be computed as follows:
\begin{align*}
	\phi(\pi,n^2) &= \cos(n\pi) + \int_0^\pi K(\pi,t)\cos(nt)\,dt
	\\
	&= (-1)^n + \int_0^\pi K(\pi,t)\cos(nt)\,dt,
	\\
\end{align*}
and
\begin{align*}
	\phi'(\pi,n^2) &= -n\sin(n\pi) + K(\pi,\pi)\cos(n\pi) + \int_0^\pi \del_x K(\pi,t)\cos(nt)\,dt
	\\
	&= (-1)^nK(\pi,\pi) + \int_0^\pi \del_x K(\pi,t)\cos(nt)\,dt.
\end{align*}
Since the quotient $\f{\phi'(\pi,n^2)}{\phi(\pi,n^2)}$ is independent of $n$ by \cite[Th. 2.10.5]{Levitan}, one has
\begin{align*}
	H = -\f{\phi'(\pi,n^2)}{\phi(\pi,n^2)} = -\lim_{n\to\infty} \f{\phi'(\pi,n^2)}{\phi(\pi,n^2)}.
\end{align*}
Moreover, by the Riemann-Lebesgue Lemma one has
\begin{align*}
	&\int_0^\pi K(\pi,t)\cos(nt)\,dt \xrightarrow{n\to\infty} 0,
	\\
	&\int_0^\pi \del_x K(\pi,t)\cos(nt)\,dt \xrightarrow{n\to\infty} 0.
\end{align*}
We therefore conclude that
\begin{align}
	H &= -\lim_{n\to\infty} \f{(-1)^nK(\pi,\pi) + \int_0^\pi \del_x K(\pi,t)\cos(nt)\,dt}{(-1)^n + \int_0^\pi K(\pi,t)\cos(nt)\,dt}
	\nonumber
	\\
	&= -K(\pi,\pi).
	\label{eq:H_finite_data}
\end{align}
Thus, $H$ can be computed in finitely many arithmeric operations and the proof is complete.
\qed

\section{Proof of Theorem \ref{thm:main2}\eqref{eq:mainth2}}
There are two distinct aspects to the proof. To prove the $\Delta_2^A$ result we must demonstrate that there exists an arithmetic algorithm which computes $q, h$ and $H$ in one limit.  To prove the $\notin \Delta_1^G$ result we  construct a counterexample which shows that error control cannot possibly hold. We begin by stating a classical lemma, which will be used multiple times in the sequel.
\begin{lemma}\label{lemma:T_inverse_bound}
	Let $\mathcal{B}$ be a Banach space. If $A,T:\mathcal{B}\to \mathcal{B}$ are bounded and invertible and $\|A-T\|<\|A^{-1}\|^{-1}$, then 
	\begin{align}\label{eq:T_inverse_bound}
		\|T^{-1}\| &\leq \f{\|A^{-1}\|}{1-\|A^{-1}\| \|A-T\|}.
	\end{align}
\end{lemma}
\begin{proof}
	This is classical, see for example \cite[Ch. I, Eq. (4.24)]{Kato}.
\end{proof}

\subsection{Proof of $(\Omega,\Lambda,\mathcal{M}_p,\Xi_p)\in \Delta_2^A$}
This proof follows a simple trajectory: we show that by letting $N\to+\infty$ in Theorem \ref{thm:main2}\eqref{eq:mainth1} we obtain the desired result. To this end we first introduce some useful notation. For $(\lambda,\alpha)\in\Omega\subset\R^{N_0}\times\R_+^{N_0}$ and $N\in\N$ define
	\begin{align*}
	\lambda_{n,N}
	:=
	\begin{cases}
	\lambda_n & n\leq N\\
	n^2 & n> N,
	\end{cases}\\
	\alpha_{n,N}
	:=
	\begin{cases}
	\alpha_n & n\leq N\\
	\f\pi2 & n> N,
	\end{cases}
	\end{align*}
to be the input of the finite data algorithm in the proof of Theorem \ref{thm:main2}\eqref{eq:mainth1}. We  write $F_N$ for the function defined in \eqref{eq:F_finite} (finite sum) and $F$ for the function defined in \eqref{eq:F_def} (infinite sum). Accordingly, we let $K_N$ denote the solution of \eqref{eq:integral_equation} with $F_N$ in lieu of $F$, i.e.
\begin{equation*}
	K_N(x,y) + F_N(x,y) + \int_0^x K_N(x,t)F_N(t,y)\,dt = 0,
\end{equation*}
 and $K$ denote the solution of \eqref{eq:integral_equation} with $F$ as before.
By \cite[Lemma 2.2.2]{Levitan} one has $F_N(x,y)\to F(x,y)$ for all $x,y\in[0,\pi]$ and $\|F_N\|_\infty<C$ for all $N\in\N$. In analogy with  Notation \ref{notation} and \eqref{eq:int-ker} we define
\begin{notation}\label{notation2}
	Quantities for $N=+\infty$:
	\begin{itemize}
		\item $f_x(y) := -F(x,y)$,
		\item $u_x(y) := K(x,y)$, 
		\item $k(t,s) := -F(s,t)$,
		\item $(\kappa_x u)(t) := \int_0^x k(t,s)u(s)\,ds$.
	\end{itemize}
	Quantities for $N$ finite:
	\begin{itemize}
		\item $f_{x,N}(y) := -F_N(x,y)$,
		\item $u_{x,N}(y) := K_N(x,y)$, 
		\item $k_N(t,s) := -F_N(s,t)$,
		\item $(\kappa_{x,N} u)(t) := \int_0^x k_N(t,s)u(s)\,ds$.
	\end{itemize}
\end{notation}
\begin{lemma}
	\label{lemma:k_convergence}
	For every $p\in[1,+\infty)$ one has
	\begin{enumi}
		\item $k_N\to k$ in $L^p([0,\pi]^2)$,
		\item $k_N(s,\cdot)\to k(s,\cdot)$ in $L^p([0,\pi])$ for every $s\in[0,\pi]$,
		\item $f_{x,N}\to f_x$ in $L^p([0,\pi])$ for every $x\in[0,\pi]$.
	\end{enumi}
\end{lemma}
\begin{proof}
	Follows immediately from the bounded convergence $F_N(x,y)\to F(x,y)$  (cf. \cite[Lemma 2.2.2]{Levitan}) and the dominated convergence theorem.
\end{proof}
\begin{lemma}
	\label{lemma:k_resolvent_bounded}
	The operator $(I-\kappa_x)^{-1}$ is bounded from $L^p([0,x])$ to $L^p([0,x])$ for all $x\in[0,\pi]$ and all $p\in[2,+\infty]$.
\end{lemma}
\begin{proof}
	Fix $x\in[0,\pi]$. By \cite[Th. 2.3.1]{Levitan} the operator $(I-\kappa_x)^{-1}$ exists as an operator from $L^2([0,x])$ to $L^2([0,x])$. Boundedness follows from the open mapping theorem.
Now let $f\in L^p([0,x])$ with $p\in [2,+\infty]$. Then by H\"older $f$ also belongs to  $L^2([0,x])$. Hence by the above, there exists a solution $u_x\in L^2([0,x])$, i.e.
	\begin{align*}
		u_x - \int_0^x k(\cdot,s)u_x(s)\,ds = f_x.
	\end{align*}
	Using H\"older's inequality again, this implies
	\begin{align*}
		\|u_x\|_{L^p([0,x])} &\leq \|f_x\|_{L^p([0,x])} + \pi^{\f12+\f1p}\|k\|_{L^\infty([0,\pi]^2)}\|u_x\|_{L^2([0,x])}
		\\
		&\leq \|f_x\|_{L^p([0,x])} + \pi^{\f12+\f1p}\|k\|_{L^\infty([0,\pi]^2)}\|(I-\kappa_x)^{-1}\|_{L^2\to L^2}\|f_x\|_{L^2([0,x])}
	\end{align*}
	hence $u_x\in L^p([0,x])$ and $(I-\kappa_x)^{-1}:L^p([0,x])\to L^p([0,x])$ is bounded.
\end{proof}
\begin{corollary}\label{cor:sup_bounded}
	For $p\in[2,+\infty]$ one has 
	\begin{align*}
		\sup_{x\in[0,\pi]} \left\|(I-\kappa_x)^{-1}\right\|_{L^p\to L^p} <+\infty.
	\end{align*}
\end{corollary}
\begin{proof}
	Follows from continuity of the map $x\mapsto \|(I-\kappa_x)^{-1}\|_{L^p\to L^p}$ (cf. \cite[Lemma 2.3.1]{Levitan}) and compactness of the interval $[0,\pi]$.
\end{proof}
\begin{lemma}
	\label{lemma:kN_resolvent_uniformly_bounded}
	For  $p\in[2,+\infty)$ one has
	\begin{align*}
		\limsup_{N\to\infty}\sup_{x\in[0,\pi]} \left\|(I-\kappa_{x,N})^{-1}\right\|_{L^p([0,x])\to L^p([0,x])} < +\infty.
	\end{align*}
\end{lemma}
\begin{proof}
	For brevity we denote $\|\cdot\|:=\|\cdot\|_{L^p([0,x]) \to L^p([0,x])}$. By Lemma \ref{lemma:T_inverse_bound} one has the bound
	\begin{align}\label{eq:kN_resolvent_bound}
		\|(I-\kappa_{x,N})^{-1}\| \leq \f{\|(I-\kappa_x)^{-1}\|}{1-\|\kappa_x-\kappa_{x,N}\|\|(I-\kappa_x)^{-1}\|}
	\end{align}
	if $\|\kappa_x-\kappa_{x,N}\|<\|(I-\kappa_x)^{-1}\|^{-1}$. Thus by Lemma \ref{lemma:k_resolvent_bounded} it suffices to estimate $\|\kappa_x-\kappa_{x,N}\|$. To this end, let $u\in L^p([0,x])$. One has
	\begin{align*}
		\|\kappa_xu - \kappa_{x,N}u\|_{L^p}^p &= \int_0^x\left|\int_0^x (k(t,s)-k_N(t,s))u(s)\,ds\right|^p\,dt
		\\
		&\leq \int_0^x\left(\|k(t,\cdot)-k_N(t,\cdot)\|_{L^{p'}}\|u\|_{L^p}\right)^p\,dt
		\\
		&= \|u\|_{L^p}^p \int_0^x\|k(t,\cdot)-k_N(t,\cdot)\|_{L^{p'}}^p\,dt 
		\\ 
		&= \|u\|_{L^p}^p \|k-k_{N}\|_{L^p([0,x];L^{p'}([0,x]))}^{p},
	\end{align*}
	where $p^{-1}+p'^{-1}=1$. Therefore
	\begin{align*}
		\sup_{x\in[0,\pi]} \|\kappa_x-\kappa_{x,N}\|_{L^p\to L^p} \leq \|k-k_{N}\|_{L^p([0,\pi];L^{p'}([0,\pi]))}
	\end{align*}
	Next, use Lemma \ref{lemma:k_convergence} (ii) to conclude that
	\begin{align*}
		\limsup_{N\to\infty}\sup_{x\in[0,\pi]} \|\kappa_x-\kappa_{x,N}\|_{L^p\to L^p} &\leq \limsup_{N\to\infty}\|k-k_N\|_{L^p([0,\pi];L^{p'}([0,\pi]))}
		\\
		&= 0.
	\end{align*}
 	Together with \eqref{eq:kN_resolvent_bound} and Corollary \ref{cor:sup_bounded} this implies the assertion.
\end{proof}
After this preparation, we are ready to prove convergence of $u_{x,N}$ to $u_x$. 
\begin{prop}\label{prop:u_convergence}
	For all $x\in[0,\pi]$ and for all $p\in[2,+\infty)$, one has $u_{x,N}\to u_x$ in $L^p([0,x])$ as $N\to+\infty$.
\end{prop}
\begin{proof}
	Let $x\in[0,\pi]$ and note that $f_{x,N},f_x\in L^p([0,x])$ for all $p\in[2,+\infty)$, $N\in\N$. By \eqref{eq:integral_equation2} one has
	\begin{align}
		u_{x,N} - u_x &= f_{x,N} - f_x + \int_0^x \big(k_N(\cdot,s)u_{x,N}(s) - k(\cdot,s)u_{x}(s)\big)\,ds
		\nonumber
		\\
		&= f_{x,N} - f_x + \int_0^x  (k_N(\cdot,s)-k(\cdot,s))u_{x,N}(s) + k(\cdot,s)(u_{x,N}(s)-u_x(s))  \,ds
		\label{eq:integral_equation_difference}
	\end{align}
	Rearranging terms we have
	\begin{align}\label{eq:uN-u_equation}
		(I-\kappa_x)(u_{x,N}-u_x) = f_{x,N} - f_x + \int_0^x (k_N(\cdot,s) - k(\cdot,s))u_{x,N}(s)\,ds
	\end{align}
	and hence
	\begin{align}\label{eq:(I+k)_bound}
		\|(I-\kappa_x)(u_{x,N}-u_x)\|_{L^p} &\leq \|f_{x,N} - f_x\|_{L^p} + \|u_{x,N}\|_{L^p} \|k_N-k\|_{L^p\left([0,x];L^{p'}([0,x])\right)},
	\end{align}
	where $p^{-1}+p'^{-1}=1$. By Lemma \ref{lemma:k_convergence} (ii), (iii) and Lemma \ref{lemma:kN_resolvent_uniformly_bounded} the right-hand side of \eqref{eq:(I+k)_bound} tends to $0$ as $N\to+\infty$. Finally, since $(I-\kappa_x)^{-1}$ is bounded (cf. Lemma \ref{lemma:k_resolvent_bounded}) it follows that $u_{x,N}-u_x\to 0$ in $L^p([0,x])$.
\end{proof}
We are finally ready to complete the proof of Theorem \ref{thm:main2}\eqref{eq:mainth3}. We  proceed by first proving that $K_N(x,x)=u_{x,N}(x)$ converges to $K(x,x)$ pointwise and then employ the dominated convergence theorem to prove $L^p$ convergence.

	Going back to \eqref{eq:integral_equation_difference} one has
	\begin{align*}
		u_{x,N}(x) - u_x(x) &= f_{x,N}(x) - f_x(x) + \int_0^x \big(k_N(x,s)u_{x,N}(s) - k(x,s)u_{x}(s)\big)\,ds
	\end{align*}
	for every $x\in[0,\pi]$. Hence by the triangle and H\"older's inequalities
	\begin{align*}
		|u_{x,N}(x) - u_x(x)| &\leq |f_{x,N}(x) - f_x(x)| 
		\\
		& \qquad + \|u_{x,N}\|_{L^2}\|k_N(x,\cdot) - k(x,\cdot)\|_{L^2} + \|u_{x,N}-u_x\|_{L^2}\|k(x,\cdot)\|_{L^2}
	\end{align*}
	The right-hand side tends to $0$ by Lemma \ref{lemma:k_convergence} and Proposition \ref{prop:u_convergence} for any fixed $x\in[0,\pi]$. Thus the function $x\mapsto u_{x,N}(x)$ converges to $u_x(x)$ pointwise on $[0,\pi]$. A similar argument as above can be used to prove boundedness. Indeed, by \eqref{eq:integral_equation2} we have
	\begin{align*}
		|u_{x,N}(x)| &= \left|f_{x,N}(x) + \int_0^x k_N(x,s)u_{x,N}(s) \,ds\right|
		\\
		&\leq |f_{x,N}(x)| + \pi^{\f12}\|k_N(x,\cdot)\|_\infty \|u_{x,N}\|_{L^2([0,x])}
		\\
		&\leq |f_{x,N}(x)| + \pi^{\f12}\|k_N(x,\cdot)\|_\infty \|(I-\kappa_{x,N})^{-1}\|_{L^2\to L^2}\|f_{x,N}\|_{L^2([0,x])}.
	\end{align*}
	Reverting back to the notation from \eqref{eq:integral_equation} this becomes
	\begin{align*}
		|K_N(x,x)| &\leq |F_N(x,x)| + \pi^{\f12} \|F_N(x,\cdot)\|_\infty \|(I-\kappa_{x,N})^{-1}\|_{L^2\to L^2}\|F_N(x,\cdot)\|_{L^2([0,x])}
		\\
		&\leq \|F_N\|_\infty \bigg(1+\pi^{\f32} \|F_N\|_\infty \sup_{x\in[0,\pi]}\|(I-\kappa_{x,N})^{-1}\|_{L^2\to L^2}\bigg)
		\\
		\Rightarrow \quad \|K_N\|_\infty &\leq \|F_N\|_\infty \bigg(1+\pi^{\f32} \|F_N\|_\infty \sup_{x\in[0,\pi]}\|(I-\kappa_{x,N})^{-1}\|_{L^2\to L^2}\bigg)
	\end{align*}
	Now by \cite[Lemma 2.2.2]{Levitan} and \cite[Lemma I.9.1]{MR0369797} $F_N$ converges \emph{boundedly} to $F$, i.e. there exists a constant $C_F$ such that $\|F_N\|_\infty\leq C_F$ for all $N\in\N$. Consequently $\|K_N\|_\infty$ is uniformly bounded:
	\begin{align*}
		  \limsup_{N\to\infty}\|K_N\|_\infty &\leq C_F \bigg(1+\pi^{\f32} C_F \limsup_{N\to\infty}\sup_{x\in[0,\pi]}\|(I-\kappa_{x,N})^{-1}\|_{L^2\to L^2}\bigg).
	\end{align*}
	Combining this fact with the pointwise convergence $K_N(x,x)\to K(x,x)$, the dominated convergence theorem implies that $\|K_N(\cdot,\cdot)-K(\cdot,\cdot)\|_{L^p}\to 0$ for all $p\in[1,+\infty)$, where $K(\cdot,\cdot)$ denotes the function $x\mapsto K(x,x)$. The following calculation concludes the reconstruction of the potential.
	\begin{equation}\label{eq:q-error-calculation}
	\begin{split}
		\|q-q_N\|_{W^{-1,p}([0,\pi])} &= \sup_{\substack{\phi\in W^{1,p'}_0([0,\pi]) \\ \|\phi\|=1}} \left|\int_0^\pi (q(x)-q_N(x))\phi(x)\,dx\right|
		\\
		&= \sup_{\substack{\phi\in W^{1,p'}_0([0,\pi]) \\ \|\phi\|=1}} \left|\int_0^\pi 2\f{d}{dx}(K(x,x)-K_N(x,x))\phi(x)\,dx\right|
		\\
		&= \sup_{\substack{\phi\in W^{1,p'}_0([0,\pi]) \\ \|\phi\|=1}} \left|\int_0^\pi 2(K(x,x)-K_N(x,x))\phi'(x)\,dx\right|
		\\
		&\leq \sup_{\substack{\phi\in W^{1,p'}_0([0,\pi]) \\ \|\phi\|=1}} 2\|K(\cdot,\cdot)-K_N(\cdot,\cdot)\|_{L^p([0,\pi])} \|\phi'\|_{L^{p'}([0,\pi])}
		\\
		&\leq 2\|K(\cdot,\cdot)-K_N(\cdot,\cdot)\|_{L^p([0,\pi])} 	
		\\
		&\xrightarrow{N\to\infty} 0.
	\end{split}
	\end{equation} 
	The proof is completed by reconstructing the boundary conditions $h$, $H$. To reconstruct $h$, simply note that $h=K(0,0)=\lim_{N\to\infty}K_N(0,0)$. To reconstruct $H$, recall the expressions \eqref{eq:asymptotics} and \eqref{eq:omega}, as well as the expression \eqref{eq:q_formula} relating $q$ and $K$, all of which together  imply that 
	\begin{equation}
	\begin{split}
		h+H+K(\pi,\pi)-K(0,0) &= \pi\Big((n+1)\lambda_n^{\f12} - (n+1)n - \kappa_n\Big)
		\label{eq:H_infinite_data}
		\\
		\Leftrightarrow \qquad H &= \lim_{n\to\infty}\pi((n+1)\lambda_n^{\f12} - (n+1)n) - h - K(\pi,\pi)+K(0,0)
		\\
		\Leftrightarrow \qquad H &= \lim_{n\to\infty}\pi((n+1)\lambda_n^{\f12} - (n+1)n)- K(\pi,\pi)
		\\
		\Leftrightarrow \qquad H &= \lim_{n\to\infty}\big[\pi((n+1)\lambda_n^{\f12} - (n+1)n) - K_n(\pi,\pi)\big]
	\end{split}
	\end{equation}
	This completes the proof that $(\Omega,\Lambda,\mathcal{M}_p,\Xi_p)\in \Delta_2^A$.\qed
\subsection{Proof of $(\Omega,\Lambda,\mathcal{M}_p,\Xi_p)\notin \Delta_1^G$}
This proof is done by contradiction. Assume that there exists a sequence of (general) algorithms $\{\Gamma_N\}_{N\in\N}$, each with output $(q_N,h_N,H_N)_{N\in\N}$, which approximates $(q,h,H)$ in the space $W^{-1,p}([0,\pi])\times \R\times\R$ with explicit error control, i.e. for all $(\lambda,\alpha)\in\Omega$ and all $N\in\N$ one has 
\begin{align*}
	\|q_N - q\|_{W^{-1,p}([0,\pi])} &< 2^{-N},
	\\
	|h_N - h| &< 2^{-N},
	\\
	|H_N - H| &< 2^{-N}.
\end{align*}
In order to derive a contradiction, consider the trivial sequences $\lambda_n = n^2$ for all $n\in\N_0$, $\alpha_n = \f\pi 2$ for $n\geq 1$, $\alpha_0 = \pi$. Clearly, the corresponding potential and boundary conditions are $q\equiv 0$ on $[0,\pi]$, $h=H=0$. By assumption, for all $N\in\N$ we have
\begin{align}\label{eq:contradiction_hyp}
\begin{split}
	\|q_N\|_{W^{-1,p}([0,\pi])} &< 2^{-N},
	\\
	|h_N| &< 2^{-N},
	\\
	|H_N| &< 2^{-N}.
\end{split}
\end{align}
 It shall be enough for our purposes to consider the case $N=1$. By definition of an algorithm, the action of $\Gamma_{1}$, say, can only depend on a finite subset $\Lambda_{\Gamma_1}(\lambda,\alpha)\subset\Lambda$, say a subset of $\{(\lambda_n,\,\alpha_n) \,|\, n\leq n_0-1\}$. We will now prove that a change in the norming constant $\alpha_{n_0}$ necessarily induces a large change in $h$. Note that altering $\alpha_{n_0}$ cannot possibly change the output $(q_{1},h_{1},H_{1})$ because of the consistency requirement $\Lambda_{\Gamma_1}(\lambda,\alpha)=\Lambda_{\Gamma_1}(\tilde\lambda,\tilde\alpha)$ whenever $\lambda_j = \tilde\lambda_j$ and $\tilde\alpha_j = \tilde\alpha_j$ for $j=1,\dots, n_0-1$.

Now consider the spectral data $(\tilde\lambda,\tilde\alpha)$ given by
\begin{align*}
	\tilde\lambda_n &= \lambda_n \quad \text{ for all }n\in\N_0
	\\
	\tilde\alpha_n &= \begin{cases}
		\alpha_n, &\text{for } n\neq n_0
		\\
		\f{\pi}{\pi+2}, &\text{for } n = n_0
	\end{cases}
\end{align*}
This choice implies
\begin{align*}
	F(x,y) &= \Big(\f{1}{\alpha_{n_0}} - \f2\pi\Big)\cos(n_0x)\cos(n_0y)
	\\
	&= \Big(\f{\pi+2}{\pi} - \f2\pi\Big) \cos(n_0x)\cos(n_0y)
	\\
	&= \cos(n_0x)\cos(n_0y)
\end{align*}
where $F$ was defined in \eqref{eq:F_def}. Using \eqref{eq:BCs} we conclude that the corresponding left-hand boundary condition is
\begin{align*}
	\tilde h = -F(0,0) = -1
\end{align*}
and thus $|\tilde h-h_N|>\f12$, contradicting \eqref{eq:contradiction_hyp}. Therefore, no sequence of algorithms $\{\Gamma_N\}_{N\in\N}$ as in our assumption can exist.
\qed

\section{Proof of Theorem \ref{thm:main2}\eqref{eq:mainth3}}
In this section we prove that on the set $\Omega_{0,M}$ it is possible to devise a sequence of arithmetic algorithms that has guaranteed error bounds. The idea is that in the expression \eqref{eq:F_def} for $F$ we want to quantify how close to each other are terms of the form $\cos(\lambda_n^{\f12}x)$ and $ \cos(nx)$. In Section \ref{sec:riesz_theory} with some abstract results about Riesz bases which are ``close'' to one another, which are then applied to our problem in Section \ref{sec:delta1}.

\subsection{Preliminary facts regarding Riesz bases}\label{sec:riesz_theory}
	
In this section we let $\h$ be a separable Hilbert space with scalar product $\langle\cdot,\cdot\rangle$ and norm $\|\cdot\|=\sqrt{\langle\cdot,\cdot\rangle}$, and let $\{f_n\}_{n\in\N_0}$ be an orthonormal basis for $\h$. Moreover, let $\{g_n\}_{n\in\N_0}\subset\h$ satisfy the following hypothesis.
\begin{hyp}\label{hyp:Omega}
	For the rest of this subsection, assume that
	\begin{enumi}
		\item 
		$\{f_n\}_{n\in\N_0}$ and $\{g_n\}_{n\in\N_0}$ are $\ell^2$-close, i.e.
		\begin{align*}
			\Omega &:= \bigg( \sum_{n=0}^\infty \|g_n-f_n\|^2\bigg)^\f12 <+\infty.
		\end{align*}
		\item
		There exist constants $\eta,\mu>0$ and $J_0\in\N$ such that for any  $J>J_0$ and $j,k\in \N_0$ with $j\leq J<k$  one has
		\begin{align*}
			|\langle g_j, f_k \rangle| \leq \f{\eta}{k^2-(j+\f \mu{j+1})^2}.
		\end{align*}
	\end{enumi}
\end{hyp}
By \cite[Prop. 1.8.5]{FreilingYurko}, Hypothesis \ref{hyp:Omega}(i) implies that $\{g_n\}_{n\in\N_0}$ is a Riesz basis. 
The goal of this subsection is to prove explicit computable bounds for the Riesz basis $\{g_n\}_{n\in\N_0}$ (cf. Proposition \ref{prop:C_Riesz} below). To this end we define an operator $T:\h\to \h$ by
\begin{align*}
	Tf_n := g_n
\end{align*}
for $n\in\N_0$. Because $\{g_n\}_{n\in\N_0}$ forms a Riesz basis, $T$ is boundedly invertible in $\cH$.
Thus for arbitrary $u\in \cH$
\begin{align}
	\label{eq:general_riesz_bound}
\begin{aligned}
	\|u\|^2 &= \|(T^*)^{-1}T^*u\|^2 \leq \|(T^*)^{-1}\|^2 \|T^*u\|^2
	\\ 
	&= \|T^{-1}\|^2 \sum_{j=0}^\infty |\langle T^*u,f_j\rangle|^2
	= \|T^{-1}\|^2 \sum_{j=0}^\infty |\langle u,Tf_j\rangle|^2
	= \|T^{-1}\|^2 \sum_{j=0}^\infty |\langle u,g_j\rangle|^2.
\end{aligned}
\end{align}
We are going to derive a computable bound for  $\|T^{-1}\|$, the operator norm of $T^{-1}$ which is defined in the standard way. Define 
\begin{align*}
	\Omega_J &:= \bigg( \sum_{n=J+1}^\infty \|g_n-f_n\|^2\bigg)^\f12.
\end{align*}
The matrix representation of $T$ in the basis $\{f_j\}_{j\in\N_0}$ has the form
\begin{align}\label{eq:T_ij_def}
	(T_{ij}) = \begin{pmatrix}
		\langle g_0,f_0\rangle & \langle g_1,f_0\rangle & \cdots
		\\
		\langle g_0,f_1\rangle & \langle g_1,f_1\rangle & \cdots
		\\
		\vdots & \vdots & \ddots
	\end{pmatrix}
\end{align}
Let $J\in\N$ and decompose this matrix into 4 blocks
\begin{align}\label{eq:T_block}
	(T_{ij}) = \begin{pmatrix}
		A_J & B_J
		\\
		C_J & D_J
	\end{pmatrix},
\end{align}
where $A_J = (T_{ij})_{i,j=0}^J$, i.e.
\begin{align*}
	A_J = \begin{pmatrix}
		\langle g_0,f_0\rangle & \cdots & \langle g_J,f_0\rangle
		\\
		\vdots & \ddots & \vdots
		\\
		\langle g_0,f_J\rangle & \cdots & \langle g_J,f_J\rangle
	\end{pmatrix}
\end{align*}
and $B_J,C_J,D_J$ are defined in the obvious way so that \eqref{eq:T_block} holds. 
\begin{lemma}\label{lemma:T-A_bound}
	If Hypothesis \ref{hyp:Omega} is satisfied, then for $J>2\mu$ one has
	\begin{align}
		\left\| T - \begin{pmatrix}
			A_J & 0 \\ 0 & \mathrm{Id}
		\end{pmatrix}\right\| \leq \max\Big\{ 4\eta\f{\log(J)}{J},\Omega_J \Big\}.
	\end{align}
\end{lemma}
\begin{proof}
	We first estimate the operator norms $\|B_J\|$ and $\|D_J-\mathrm{Id}\|$ using Hypothesis \ref{hyp:Omega}(i) and then focus on $\|C_J\|$. Explicit calculations of the Hilbert-Schmidt norms give
	\begin{align*}
		\|B_J\|^2 &\leq \|B_J\|_{\mathrm{HS}}^2 = \sum_{k=0}^J\sum_{j=J+1}^\infty |\langle g_j,f_k \rangle|^2
		= \sum_{k=0}^J\sum_{j=J+1}^\infty |\langle g_j-f_j,f_k \rangle|^2
		\\
		&\leq \sum_{k=0}^\infty\sum_{j=J+1}^\infty |\langle g_j-f_j,f_k \rangle|^2
		= \sum_{j=J+1}^\infty \|g_j-f_j\|^2
		= \Omega_J^2
	\end{align*}
and
	\begin{align*}
		\|D_J-\mathrm{Id}\|^2 &\leq \|D_J-\mathrm{Id}\|_{\mathrm{HS}}^2 = \sum_{j=J+1}^\infty\sum_{k=J+1}^\infty |\langle g_j,f_k \rangle - \delta_{jk}|^2\qquad\qquad
		\\
		&= \sum_{j=J+1}^\infty\sum_{k=J+1}^\infty |\langle g_j-f_j,f_k \rangle + \langle f_j,f_k \rangle - \delta_{jk}|^2
		\\
		&= \sum_{j=J+1}^\infty\sum_{k=J+1}^\infty |\langle g_j-f_j,f_k \rangle|^2
		\\
		&\leq \sum_{j=J+1}^\infty\sum_{k=0}^\infty |\langle g_j-f_j,f_k \rangle|^2
		\\
		&= \sum_{j=J+1}^\infty \|g_j-f_j\|^2
		\\
		&= \Omega_J^2.
	\end{align*}
	Next we estimate $\|C_J\|$. To this end, we use a general result, which follows from the Riesz-Thorin interpolation theorem. For any infinite matrix $\cO$ one has
	\begin{align}\label{eq:O<X1X2}
		\|\cO\| \leq \max\{X_1,X_2\},
	\end{align}
	where
	\begin{align*}
		X_1 = \sup_{k\in\N_0}\sum_{j=0}^\infty |\cO_{kj}|,
		\qquad
		X_2 = \sup_{j\in\N_0}\sum_{k=0}^\infty |\cO_{kj}|.
	\end{align*}
	Setting $\cO:=\begin{pmatrix} 0 & 0 \\ C_J & 0 \end{pmatrix}$ and using Hypothesis \ref{hyp:Omega}(ii) we compute
	\begin{align*}
		X_1 &= \sup_{k\geq J+1}\sum_{j=0}^J |\langle g_j,f_k\rangle| 
		\leq \sup_{k\geq J+1}\sum_{j=0}^J \f{\eta}{k^2-(j+\f \mu{j+1})^2}
		= \sum_{j=0}^J \f{\eta}{(J+1)^2-(j+\f \mu{j+1})^2}
		\\
		&= \sum_{j=0}^J \f{\eta}{(J+1+j+\f \mu{j+1})(J+1-j-\f \mu{j+1})}
		\leq \sum_{j=0}^J \f{\eta}{J(J+1-j-\f \mu{j+1})}
		\\
		&= \f{\eta}{J} \sum_{j=0}^J \f{1}{J+1-(J-j)-\f{\mu}{J-j+1}}
		= \f{\eta}{J} \sum_{j=0}^J \f{1}{1+j-\f{\mu}{J+1-j}} 
		\leq \f{\eta}{J} \sum_{j=0}^J \f{2}{j+1} 
		\leq 4\eta\f{\log(J)}{J},
	\end{align*}
where we have used the inequality $x-\f{\mu}{J+1-x}\geq \f x2$, which holds for $J\geq 2\mu$ and $1\leq x\leq J$. For $X_2$ we estimate
	\begin{align*}
		X_2 &= \sup_{0\leq j\leq J}\sum_{k=J+1}^\infty |\langle g_j,f_k\rangle| 
		\leq \sup_{0\leq j\leq J}\sum_{k=J+1}^\infty \f{\eta}{k^2-(j+\f \mu{j+1})^2} 
		 = \sum_{k=J+1}^\infty \f{\eta}{k^2-(J+\f \mu{J+1})^2} 
		 \\
		 &\leq \sum_{k=J+1}^\infty \f{\eta}{k^2-(J+\f12)^2}
		 = \sum_{k=1}^\infty \f{\eta}{(k+J)^2-(J+\f12)^2}
		 = \sum_{k=1}^\infty \f{\eta}{k^2+(2k-1)J-\f14}
		 \\
		 &= \sum_{k=1}^J \f{\eta}{k^2+(2k-1)J-\f14} + \sum_{k=J+1}^\infty \f{\eta}{k^2+(2k-1)J-\f14}
		 \leq \sum_{k=1}^J \f{\eta}{(2k-1)J} + \sum_{k=J+1}^\infty \f{\eta}{k^2}
		 \\
		 &\leq \f{\eta}{J} \sum_{k=1}^J \f{1}{2k-1} + \sum_{k=J+1}^\infty \f{\eta}{k^2}
		\leq \f{2\eta}{J} \log(J) + \f\eta J
		\leq 3\eta\f{\log(J)}{J}.
%		 \leq \sum_{k=1}^N \f{c}{(2k-\f32)N} + \sum_{k=N-1}^\infty \f{c}{k^2}
%		 \\
%		 &\leq \sum_{k=1}^N \f{4c}{(4k-3)N} + \f{c}{N-1}
%		 \leq \f{4c}{N} \sum_{k=1}^N \f{1}{k} + \f{c}{N-1}
%		 \\
%		 &\leq 4c\f{\log(N)}{N} + \f{c}{N-1}
%		 \\
%		 &\leq 8c\f{\log(N)}{N}.
	\end{align*}
	Hence by \eqref{eq:O<X1X2} we have
	\begin{align*}
		\|C_J\| \leq 4\eta\f{\log(J)}{J}.
	\end{align*}
	Therefore
	\begin{align*}
		\left\| T - \begin{pmatrix}
			A_J & 0 \\ 0 & \mathrm{Id}
		\end{pmatrix} \right\| 
		&= \left\|\begin{pmatrix}
			0 & B_J \\ C_J & D_J-\mathrm{Id}
		\end{pmatrix}\right\|
		\\
		&\leq \max\big\{ \|C_J\|,\,\|B_J\|, \|D_J-\mathrm{Id}\| \big\}
		\\
		&=  \max\Big\{ 4\eta\f{\log(J)}{J},\Omega_J \Big\}.
	\end{align*}
\end{proof}
Based on Lemma \ref{lemma:T-A_bound}, we define the small parameter
\begin{align*}
	\delta_J := \max\Big\{ 4\eta\f{\log(J)}{J},\Omega_J \Big\}
\end{align*}
The following is the main result of this subsection, which follows easily. We first observe that  $\limsup_{J\in\N}\|A_J^{-1}\|<+\infty$  because $T$ is invertible and $\|A_J\oplus\mathrm{Id}-T\|\to 0$ as $J\to+\infty$. Hence $\liminf_{J\in\N}\|A_J^{-1}\|^{-1}>0$.
\begin{prop}\label{prop:C_Riesz}
	For any $u\in \cH$ and $J$ large enough to ensure that $\delta_J < \|A_J^{-1}\oplus \mathrm{Id}\|^{-1}$ one has
	\begin{align*}
		\|u\|^2 \leq C^2\sum_{j=0}^\infty |\langle u,g_j \rangle|^2,
	\end{align*}
	where $C = \f{\|A_J^{-1}\oplus\mathrm{Id}\|}{1 - \delta_J\|A_J^{-1}\oplus\mathrm{Id}\|}$. Moreover, if $\delta_J < \f12(\|A_J^{-1}\oplus\mathrm{Id}\|+1)^{-1}$, then there exists a bound $C_{\textnormal{Riesz}}$ for $C$ such that 
	\begin{enumi}
		\item $ C \leq C_{\textnormal{Riesz}} \leq 2(\delta_J^{-1}+1)$,
		\item $C_{\textnormal{Riesz}}$ can be computed by an arithmetic algorithm given the numbers $\{\delta_J\}_{J\in\N_0}$ and $\{T_{ij}\}_{i,j\in\N_0}$.
	\end{enumi}
\end{prop}

\begin{proof}
Equation  \eqref{eq:general_riesz_bound} provides us with the expression
\begin{equation*}
	\|u\|^2	\leq \|T^{-1}\|^2 \sum_{j=0}^\infty |\langle u,g_j\rangle|^2,
\end{equation*}
and Lemma \ref{lemma:T_inverse_bound} provides a bound for $ \|T^{-1}\|$. Lemma \ref{lemma:T-A_bound} provides the bound  for $\|T-A_J\oplus\mathrm{Id}\|$. This gives the expression for $C$.

	A routine that computes $C_{\textnormal{Riesz}}$ is given in Algorithms \ref{alg:compute_A} and \ref{alg:compute_C}. The upper bound for $C_{\textnormal{Riesz}}$ follows from the bound $C_{\textnormal{Riesz}} \leq \f{\|A_J^{-1}\oplus\mathrm{Id}\|+1}{1-\delta_J(\|A_J^{-1}\oplus\mathrm{Id}\|+1)}$ (cf. Algorithm \ref{alg:compute_C}) and the choice $\delta_J < \|A_J^{-1}\oplus\mathrm{Id}\|^{-1}$. Note that trivially $\|A_J^{-1}\oplus\mathrm{Id}\| = \max\{\|A_J^{-1}\|, 1\}$.

\smallskip
\begin{algorithm}[H]
	\setlength{\lineskip}{1mm}
%	\DontPrintSemicolon
	\SetAlgorithmName{Algorithm}{}{}
	\SetAlgoInsideSkip{smallskip}
	\SetAlgoSkip{bigskip}
	\SetAlgoLined
	\KwIn{$A_J$}
	Initialize $k:=1$\;
	 \While{True}{
	 	Compute $B_k:=A_J^*A_J - k^{-2}I$\;
	 	\If{$B_k>0$\tcp*{$B_k>0$ can be checked by Cholesky factorization.}}{
		 		$a_J:=k$\tcp*{Since $B_{k-1}\ngtr 0$, one necessarily has $k-1\leq\|A_J^{-1}\|\leq k$.} 
		 		break out of loop\;
	 		}
	 	$k:=k+1$\;
	 }
	 \Return{$a_J$}
	 \caption{Compute $\|A_J^{-1}\oplus\mathrm{Id}\|$}
	 \label{alg:compute_A}
\end{algorithm}
\medskip

\begin{algorithm}[H]
	\setlength{\lineskip}{1mm}
%	\DontPrintSemicolon
	\SetAlgorithmName{Algorithm}{}{}
	\SetAlgoInsideSkip{smallskip}
	\SetAlgoSkip{bigskip}
	\SetAlgoLined
	\KwIn{$\{\delta_J\}_{J\in\N_0}$ and $\{T_{ij}\}_{i,j\in\N_0}$}
	Initialize $J:=1$\;
	 \While{True}{
	 	Using Algorithm \ref{alg:compute_A}, compute approximation $a_J$ of $\|A_J^{-1}\|$ with error less than $1$\;
	 	\tcp{Since $\|A_J^{-1}\|$ remains bounded as $J\to\infty$, so does $a_J$.}
%	 	Compute $\delta_N := \max\big\{ 4\eta\f{\log(N)}{N},\Omega_N \big\}$ \;
	 	\If{$\delta_J < (a_J+1)^{-1}$}{
	 	\medskip
		 		$C := \f{a_J+1}{1 - \delta_J(a_J+1)}$\;
		 		break out of loop\;
	 		}
	 	$J:= J+1$\;
	 }
	 \Return{$J$, $C$}
	 \caption{Compute $C_{\textnormal{Riesz}}$}
	 \label{alg:compute_C}
\end{algorithm}

\end{proof}

\subsection{Application to the inverse Sturm-Liouville problem}
\label{sec:delta1}

From now on we assume that $\lambda_n,\alpha_n$ satisfy \eqref{eq:asymptotics} with $\omega=0$ and $\|\kappa\|_{\ell^2},\|\tilde\kappa\|_{\ell^2}\leq M$. This will allow us to prove explicit error bounds for $\|K(\cdot,\cdot)-K_N(\cdot,\cdot)\|_{L^\infty([0,\pi])}$ thus strengthening the estimate \eqref{eq:q-error-calculation}. Our strategy is to use eq. \eqref{eq:uN-u_equation}
\begin{align*}
	(u_{x,N}-u_x) &= (I-\kappa_x)^{-1}\left[ f_{x,N} - f_x + \int_0^x (k_N(\cdot,s) - k(\cdot,s))u_{x,N}(s)\,ds \right],
\end{align*}
which implies the bound
\begin{align}\label{eq:explicit_bounds0}
\begin{aligned}
	\|u_{x,N}-u_x\|_{L^\infty([0,x])} &\leq \|(I-\kappa_x)^{-1}\|_{L^\infty\to L^\infty}
	\\
	\times& \Big( \|f_{x,N} - f_x\|_{L^\infty([0,x])}  + \pi \|k_N-k\|_{L^\infty([0,\pi]^2)} \|u_{x,N}\|_{L^2([0,x])} \Big).
\end{aligned}
\end{align}
We will now estimate every term in \eqref{eq:explicit_bounds0}  by a computable constant.
We begin with the following lemma.
\begin{lemma}\label{lemma:explicit_bounds_F-FN}
	If $\omega=0$ and $F, F_N$ are defined by \eqref{eq:F_def} and \eqref{eq:F_finite}, respectively, then there exists $C>0$ such that 
	\begin{align*}
		\|F_N-F\|_{L^\infty([0,\pi]^2)} \leq CN^{-\f12}.
	\end{align*}
	If in addition $\|\kappa\|_{\ell^2}, \|\tilde\kappa\|_{\ell^2} \leq M$, then C can be taken to be the explicit constant
	\begin{equation}\label{eq:CM1}
	C=C_M^{(1)} := M\cosh(2\pi M)\left[\Big(\f{8\pi^2}{\sqrt{6}} + 2\pi\Big)M + 5\right].
	\end{equation}
\end{lemma}
\begin{proof}
	It follows from standard trigonometric identities that $F_N(x,y) = \f12(\Phi_N(x+y)+\Phi_N(x-y))$, where
	\begin{align}\label{eq:Phi_N_def}
		\Phi_N(t) = \f{1}{\alpha_0}\cos(\lambda_0^{\f12}t) - \f1\pi + \sum_{n=1}^N \left[\f{1}{\alpha_n}\cos(\lambda_n^{\f12}t) - \f2\pi\cos(nt)\right].
	\end{align}
	Thus, to prove the lemma is suffices to prove uniform convergence of $\Phi_N$  on the interval $[-\pi,2\pi]$. Using trigonometric identities and the expressions \eqref{eq:asymptotics} we can rewrite the terms in the sum in \eqref{eq:Phi_N_def} as follows. 
	\begin{align*}
		\f{1}{\alpha_n}\cos(\lambda_n^{\f12}t) - \f2\pi\cos(nt) &= \f{1}{\alpha_n}\cos\Big(nt+\f{\kappa_n}{n+1} t\Big) - \f2\pi\cos(nt)
		\\
		&= \f{1}{\alpha_n}\cos(nt)\cos\Big(\f{\kappa_n}{n+1} t\Big) - \f{1}{\alpha_n}\sin(nt)\sin\Big(\f{\kappa_n}{n+1} t\Big) - \f2\pi\cos(nt)
		\\
		&= \left[ \f{1}{\alpha_n}\cos\Big(\f{\kappa_n}{n+1}t\Big)-\f2\pi \right] \cos(nt) - \f{1}{\alpha_n}\sin(nt)\sin\Big(\f{\kappa_n}{n+1}t\Big)
		\\
		&= \left[ \f2\pi\Big(\cos\Big(\f{\kappa_n}{n+1}t\Big)-1\Big) + \f{\tilde\kappa_n}{n+1}\cos\Big(\f{\kappa_n}{n+1}t\Big) \right] \cos(nt)
		\\
		&\qquad\qquad\qquad- \f{1}{\alpha_n}\sin(nt)\sin\Big(\f{\kappa_n}{n+1}t\Big)
	\end{align*}
	Let $\Phi$ denote the pointwise limit of $\Phi_N$ (whose existence was proved in \cite[\S2.3]{Levitan}). Then the error for given $N\in\N$ is
	\begin{align*}
		|\Phi(t)-\Phi_N(t)|
		&\leq \sum_{n=N+1}^\infty \Bigg\{\left|\f2\pi\Bigl(\cos\Bigl(\f{\kappa_n}{n+1}t\Bigr)-1\Bigr) 
		+ \f{\tilde\kappa_n}{n+1}\cos\Bigl(\f{\kappa_n}{n+1}t\Bigr)\right| |\cos(nt)| 
		\\
		&\qquad\qquad\qquad + \f{1}{\alpha_n}\Bigl|\sin(nt)\sin\Bigl(\f{\kappa_n}{n+1}t\Bigr)\Bigr|\Bigg\}
		\\
		&\leq \sum_{n=N+1}^\infty \left\{ \f2\pi\Bigl|\cos\Bigl(\f{\kappa_n}{n+1}t\Bigr)-1\Bigr| 
		+ \f{|\tilde\kappa_n|}{n+1}\Bigl|\cos\Bigl(\f{\kappa_n}{n+1}t\Bigr)\Bigr| 
		+ \f{1}{\alpha_n}\Bigl|\sin\Bigl(\f{\kappa_n}{n+1}t\Bigr)\Bigr|\right\}
		\\
		&\leq \sum_{n=N+1}^\infty \left\{\f2\pi\f{|\kappa_n|^2}{(n+1)^2}t^2 + \f{|\tilde\kappa_n|}{n+1} + \f{1}{\alpha_n}\f{|\kappa_n|}{n+1}|t|\right\}\cosh\Bigl( \f{|\kappa_n|}{n+1}|t| \Bigr)
		\\
		&\leq \cosh(|t| M) \sum_{n=N+1}^\infty \left\{\f2\pi\f{|\kappa_n|^2}{(n+1)^2}t^2 + \f{|\tilde\kappa_n|}{n+1} + \f{1}{\alpha_n}\f{|\kappa_n|}{n+1}|t|\right\},
	\end{align*}
	where in the penultimate line the bounds $|f(z)-f(0)|\leq |z|\|f'\|_{L^\infty(B_{|z|}(0))}$ and $|\cos(z)|,|\sin(z)|,|\sinh(z)|\leq \cosh(|z|)$ were used. Setting $t=2\pi$, using $\frac{1}{n+1}<\frac1n$ and using H\"older's inequality we obtain
	\begin{align}
		|\Phi(t)-\Phi_N(t)| 
		&\leq \cosh(2\pi M)\left[8\pi\|\kappa\|_{\ell^\infty}^2\sum_{n=N+1}^\infty \f{1}{n^2} + \|\tilde\kappa\|_{\ell^2}\bigg(\sum_{n=N+1}^\infty \f{1}{n^2}\bigg)^{\f12} + 2\pi\Big\|\f{\kappa}{\alpha}\Big\|_{\ell^2}\bigg(\sum_{n=N+1}^\infty \f{1}{n^2}\bigg)^{\f12}\right]
		\nonumber
		\\
		&= \cosh(2\pi M)\bigg(\sum_{n=N+1}^\infty \f{1}{n^2}\bigg)^{\f12}\left[ 8\pi\|\kappa\|_{\ell^\infty}^2 \bigg(\sum_{n=1}^\infty \f{1}{n^2}\bigg)^{\f12} + \|\tilde\kappa\|_{\ell^2} + 2(2+\pi\|\tilde\kappa\|_{\ell^\infty})\|\kappa\|_{\ell^2} \right]
		\nonumber
		\\
		&\leq \cosh(2\pi M)\bigg(\sum_{n=N+1}^\infty \f{1}{n^2}\bigg)^{\f12}\left[ \Big(\f{8\pi^2}{\sqrt{6}} + 2\pi\Big)M + 5\right]M
		\nonumber
		\\
		&\leq N^{-\f12}\cosh(2\pi M)\left[ \Big(\f{8\pi^2}{\sqrt{6}} + 2\pi\Big)M + 5\right]M
		\label{eq:Phi_explicit_bound}
	\end{align}
	The proof is concluded by noting that $\|F-F_N\|_{L^\infty([0,\pi]^2)} = \|\Phi-\Phi_N\|_{L^\infty([-\pi,2\pi])}$.
\end{proof}

Lemma \ref{lemma:explicit_bounds_F-FN} implies explicit bounds on  the terms $\|f_{x,N} - f_x\|_{L^\infty([0,x])}$ and $\|k_N-k\|_{L^\infty([0,\pi]^2)}$ in \eqref{eq:explicit_bounds0}.  
Next we apply the theory from Section \ref{sec:riesz_theory}. The following lemma shows that Hypothesis \ref{hyp:Omega} is satisfied in our situation.

\begin{lemma}\label{lemma:hyp_satisfied}
	For $n\in\N_0$ and $t\in[0,\pi]$ define
		\begin{equation}\label{eq:fn-and-gn}
		f_n(t) := \left(\f2\pi\right)^{\f12}\cos(nt)\quad\text{and}\quad g_n(t) := \alpha_n^{-\f12}\cos(\lambda_n^{\f12}t).
		\end{equation}
	Then the families $\{f_n\}_{n\in\N_0}$, $\{g_n\}_{n\in\N_0}$ satisfy Hypothesis \ref{hyp:Omega} with $\cH=L^2([0,\pi])$. In particular:
	\begin{enumi}
		\item For all $J\in\N$ one has 
			\begin{align}\label{eq:Omega_N_bound}
				\Omega_J = \bigg( \sum_{n=J+1}^\infty \|g_n-f_n\|_{L^2}^2\bigg)^\f12
				\leq \f{C_\Omega}{J},
			\end{align}
	where  $C_\Omega = M\pi\cosh(M\pi)\sqrt{\f{3}{2} \big(1+(2\pi)^2 M^2+(2+\pi M)^2\big)}$.
		\item Let $J\geq M$ and let $j\leq J<k$, then 
		\begin{align*}
			|\langle g_j,f_k\rangle_{L^2([0,\pi])}| \leq \f{c}{k^2-(j+\f M{j+1})^2}
		\end{align*}
		 where $c = (\f\pi2)^{-\f12}(\f2\pi+M)^{\f12} M(1+M)\cosh(M\pi)$.
	\end{enumi}
\end{lemma}

\begin{proof}
	Part (i) follows immediately from the term-wise inequality 
	\begin{align*}
		\|g_n-f_n\|_{L^2([0,\pi])}^2 \leq \f{\cosh(M \pi)^2}{(n+1)^2} \f{3\pi^2}{2} \Big(
				|\tilde\kappa_n|^2 + (2\pi)^2 \f{|\kappa_n|^4}{(n+1)^2}  + (2 + \pi M )^2 |\kappa_n|^2 \Big),
	\end{align*}
	which is proved by a  calculation similar to the ones found within the proof of Lemma \ref{lemma:explicit_bounds_F-FN}. We give the details below.
		For notational convenience, denote $\alpha_n^0:=\f\pi2$ for $n>0$ and $\alpha_0^0:=\pi$. By trigonometric identities we have
		\begin{align}
		\begin{aligned}
			g_n(t)-f_n(t) &= \f{1}{\sqrt{\alpha_n}} \cos\bigl(\lambda_n^{\f12} t\bigr) - \f{1}{\sqrt{\alpha_n^0}} \cos(nt)
			\\
			&= \bigg[\f{1}{\sqrt{\alpha_n}} \cos\Bigl(\f{\kappa_n}{n+1} t\Bigr) - \f{1}{\sqrt{\alpha_n^0}}\bigg]\cos(nt) - \f{1}{\sqrt{\alpha_n}}\sin(nt)\sin\Bigl( \f{\kappa_n}{n+1}t \Bigr).
			\label{eq:sqrt(alpha_n)cos}
		\end{aligned}
		\end{align}
		Moreover, 
		\begin{align*}
			\f{1}{\sqrt{\alpha_n}} &= \sqrt{\f2\pi + \f{\tilde\kappa_n}{n+1}}
			= \Big(\f2\pi\Big)^{\f12} \sqrt{1 + \f\pi2\f{\tilde\kappa_n}{n+1}}
		\end{align*}
		and hence
		\begin{align}\label{eq:sqrt(alpha_n)}
			\bigg| \f{1}{\sqrt{\alpha_n}} - \Big(\f2\pi\Big)^{\f12} \bigg| &= \Big(\f2\pi\Big)^{\f12}\bigg| \sqrt{1 + \f\pi2\f{\tilde\kappa_n}{n+1}} - 1 \bigg|
			\leq \Big(\f2\pi\Big)^{\f12}\f\pi2\f{|\tilde\kappa_n|}{n+1}
			= \Big(\f\pi2\Big)^{\f12}\f{|\tilde\kappa_n|}{n+1}.
		\end{align}
		Taking absolute values in \eqref{eq:sqrt(alpha_n)cos} and using \eqref{eq:sqrt(alpha_n)} we have 
		\begin{align*}
			|g_n(t)-f_n(t)| &\leq \bigg|\f{1}{\sqrt{\alpha_n}} \cos\Bigl(\f{\kappa_n}{n+1} t\Bigr) - \f{1}{\sqrt{\alpha_n^0}}\bigg|  + \bigg|\f{1}{\sqrt{\alpha_n}}\sin(nt)\sin\Bigl( \f{\kappa_n}{n+1}t \Bigr)\bigg|
			\\
			&\leq \biggl|  \f{1}{\sqrt{\alpha_n}} - \f{1}{\sqrt{\alpha_n^0}} \biggr| \biggl|\cos\Bigl(\f{\kappa_n}{n+1} t\Bigr)\biggr|
			+ \f{1}{\sqrt{\alpha_n^0}}\biggl| \cos\Bigl(\f{\kappa_n}{n+1} t\Bigr) - 1\biggr| 
			+ \f{1}{\sqrt{\alpha_n}} \biggl|\sin\Bigl( \f{\kappa_n}{n+1}t \Bigr)\biggr|
			\\
			&\leq \cosh(M |t|)\left\{\Big(\f\pi2\Big)^{\f12}\f{|\tilde\kappa_n|}{n+1} + \f{1}{\sqrt{\alpha_n^0}} \f{|\kappa_n|^2}{(n+1)^2} t^2 + \bigg[\f{1}{\sqrt{\alpha_n^0}} + \Big(\f\pi2\Big)^{\f12}\f{|\tilde\kappa_n|}{n+1} \bigg] \f{|\kappa_n|}{n+1}|t|\right\},
		\end{align*}
		where in the last line the bounds $|f(z)-f(0)|\leq |z|\|f'\|_{L^\infty(B_{|z|}(0))}$ and $|\cos(z)|,|\sin(z)|,|\sinh(z)|\leq \cosh(|z|)$ were used. Setting $t=\pi$ and focusing on $n\geq 1$ (i.e. $\alpha_n^0=\frac\pi2$) we obtain
		\begin{align*}
			|g_n(t)-f_n(t)| &\leq \cosh(M \pi)\left\{\Big(\f\pi2\Big)^{\f12}\f{|\tilde\kappa_n|}{n+1} 
			+ \Big(\f2\pi\Big)^{\f12} \f{|\kappa_n|^2}{(n+1)^2}\pi^2  
			+ \bigg[\Big(\f2\pi\Big)^{\f12} 
			+ \Big(\f\pi2\Big)^{\f12}\f{|\tilde\kappa_n|}{n+1} \bigg] \f{|\kappa_n|}{n+1}\pi\right\}
			\\
			&\leq \cosh(M \pi)\left\{\Big(\f\pi2\Big)^{\f12}\f{|\tilde\kappa_n|}{n+1} 
			+ \Big(\f2\pi\Big)^{\f12} \f{|\kappa_n|^2}{(n+1)^2}\pi^2  
			+ \bigg[\Big(\f2\pi\Big)^{\f12}  
			+ \Big(\f\pi2\Big)^{\f12}M \bigg] \f{|\kappa_n|}{n+1}\pi\right)
			\\
			&= \f{\cosh(M \pi)}{n+1}\Big(\f\pi2\Big)^{\f12}\Big(
				|\tilde\kappa_n| 
				+ 2\pi \f{|\kappa_n|^2}{n+1} 
				+ [2 + \pi M ] |\kappa_n| \Big)
		\end{align*}
		Squaring both sides and using the inequality $(a+b+c)^2\leq 3(a^2+b^2+c^2)$ finally yields
		\begin{align*}
			|g_n(t)-f_n(t)|^2 &\leq \f{\cosh(M \pi)^2}{(n+1)^2} \f{3\pi}{2} \Big(
				|\tilde\kappa_n|^2 + (2\pi)^2 \f{|\kappa_n|^4}{(n+1)^2}  + [2 + \pi M ]^2 |\kappa_n|^2 \Big)
			\\
			\|g_n-f_n\|_{L^2([0,\pi])}^2 &\leq \f{\cosh(M \pi)^2}{(n+1)^2} \f{3\pi^2}{2} \Big(
				|\tilde\kappa_n|^2 + (2\pi)^2 \f{|\kappa_n|^4}{(n+1)^2}  + [2 + \pi M ]^2 |\kappa_n|^2 \Big).
		\end{align*}
	
	Now we prove (ii). For brevity, denote $\rho_n:=\lambda_n^{\f12}$. Then we compute 
		\begin{align*}
		\langle g_j,f_k\rangle_{L^2([0,\pi])} &= \bigg(\f\pi2\alpha_j\bigg)^{-\f12} \int_0^\pi \cos(\rho_j t)\cos(kt)\,dt
		\\
		&= (-1)^k\bigg(\f\pi2\alpha_j\bigg)^{-\f12} \f{\rho_j  \sin(\rho_j\pi) }{\rho_j^2 - k^2}
		\\
		&= (-1)^{j+k}\bigg(\f\pi2\alpha_j\bigg)^{-\f12} \f{\rho_j  \sin(\f{\kappa_j}{{j+1}}\pi) }{\rho_j^2 - k^2}.
	\end{align*}
	Thus
	\begin{align*}
		|\langle g_j,f_k\rangle_{L^2([0,\pi])}| 
		&\leq \bigg(\f\pi2\alpha_j\bigg)^{-\f12} \f{\rho_j}{|k^2 - \rho_j^2|} \f{|\kappa_j|}{{j+1}} \Bigl|\cosh\Bigl(\f{\kappa_j}{{j+1}}\pi\Bigr)\Bigr|
		\\
		&\leq \bigg(\f\pi2\alpha_j\bigg)^{-\f12} \f{\rho_j}{{j+1}}\f{|\kappa_j|}{|k^2 - \rho_j^2|}\cosh\Bigl(\f{|\kappa_j|}{{j+1}}\pi\Bigr)
		\\
		&= \bigg(\f\pi2\alpha_j\bigg)^{-\f12} \f{j+\f{|\kappa_j|}{j+1}}{{j+1}}\f{|\kappa_j|}{|k^2-(j+\f{\kappa_j}{{j+1}})^2|}\cosh(|\kappa_j|\pi)
		\\
		&\leq \bigg(\f\pi2\alpha_j\bigg)^{-\f12} (1+M)\f{|\kappa_j|}{|k^2-(j+\f{M}{{j+1}})^2|}\cosh(M\pi).
	\end{align*}
%
%\begin{align*}
%	|k^2-(j+z)^2| &\geq ||k^2|-|j+z|^2|
%\\
%&= |k^2-|j+z|^2|
%\\[2mm]
%|j+z| &\leq |j| + |z| = j+|z| \leq j+M
%\\[2mm]
%\Rightarrow\qquad 
%|k^2-(j+z)^2| &\geq |k^2-(j+M)^2|
%\end{align*}
%

Note that as soon as $J\geq M$ one has $J+1> j+\f{M}{j+1}$ for all $j\leq J$. Therefore, since $k\geq J+1$, the denominator in the last term above is always nonzero.
%	
%	\begin{align*}
%		N+1 & > j + \f M {j+1} \qquad\text{ for all } j\leq N
%		\\
%		\text{ Choose }N&=M
%		\\
%		M+1 & > j + \f M {j+1} 
%		\\
%		(M+1)(j+1) - j(j+1) - M & > 0
%	\end{align*}
\end{proof}
The following proposition now follows immediately from Lemma \ref{lemma:hyp_satisfied} and Proposition \ref{prop:C_Riesz}.
\begin{prop}\label{lemma:explicit_Riesz_bound}
	Let $\{f_n\}_{n\in\N_0}$ and $\{g_n\}_{n\in\N_0}$ be as in \eqref{eq:fn-and-gn}. There exists a constant $C_M^{(2)}$, which can be computed in finitely many arithmetic operations from the information in $\Lambda$, such that for all $u\in L^2([0,\pi])$
	\begin{align}\label{eq:riesz_bound}
		\|u\|_{L^2([0,\pi])}^2 \leq C_M^{(2)} \sum_{n=0}^\infty \big|\langle u,g_n\rangle_{L^2([0,\pi])}\big|^2.
	\end{align}
%	for all $N > \exp\big(\f{C_\Omega}{4c}\big)$.
\end{prop}
\begin{proof}
	Lemma \ref{lemma:hyp_satisfied} shows that Hypothesis \ref{hyp:Omega} is satisfied by $\{f_n\}_{n\in\N_0}$, $\{g_n\}_{n\in\N_0}$ with $\mu=M$ and $\eta=c$. Moreover, \eqref{eq:Omega_N_bound} shows that whenever  $J > \exp\big(\f{C_\Omega}{4c}\big)$
	\begin{align*}
		\delta_J &= \max\Big\{ 4c\f{\log(J)}{J},\Omega_J \Big\}
		= 4c\f{\log(J)}{J}.
	\end{align*}
Computability of $C_M^{(2)}$ follows from Proposition \ref{prop:C_Riesz} and the fact that the matrix elements $T_{ij}$ in \eqref{eq:T_ij_def} are given by scalar products $\langle g_i,f_j\rangle_{L^2([0,\pi])}$, which can be calculated symbolically for all $i,j\in\N_0$ (recall that the $g_i$'s and $f_j$'s are all cosines).
\end{proof}

Next we prove a computable bound on the operator norm of $(I-\kappa_x)^{-1}$. 
\begin{lemma}\label{lemma:explicit_bound_(I-k)}
	Let $C_M^{(2)}$ be defined as in Lemma \ref{lemma:explicit_Riesz_bound}. Then for all $x\in[0,\pi]$ one has
	\begin{align}
		\|(I-\kappa_x)^{-1}\|_{L^2([0,x])\to L^2([0,x])} &\leq C_M^{(2)},
		\tag{i}
		\\
		\|(I-\kappa_x)^{-1}\|_{L^\infty([0,x])\to L^\infty([0,x])} &\leq \pi^{\f12}+\pi\|k\|_\infty C_M^{(2)}.
		\tag{ii}
	\end{align}
\end{lemma}
\begin{proof}
	We first prove (i). To this end, let $f_x\in L^2([0,x])$ and let $u_x = (I-\kappa_x)^{-1}f_x$ be  the solution to the integral equation \eqref{eq:integral_equation2}, which we rewrite here for convenience:
	\begin{align}\label{eq:integral_equation2_new}
		u_x - \int_0^x k(\cdot,s)u_x(s)\,ds = f_x.
	\end{align}
	Testing this equation with $u_x$ and using Parseval's identity (as detailed in the proof of \cite[Lemma 1.5.7]{FreilingYurko}) gives 
	\begin{align}\label{eq:plancherel_result}
		\sum_{n=0}^\infty \big|\langle u_x,g_n\rangle_{L^2([0,x])}\big|^2 = \langle f_x, u_x\rangle_{L^2([0,x])},
	\end{align}
	where $g_n(t) := \alpha_n^{-\f12} \cos\big(\lambda_n^{\f12}t\big)$.
	Combining \eqref{eq:plancherel_result} and \eqref{eq:riesz_bound} we define $u_x^0(y) := u_x(y)$ for $y\leq x$ and $u_x^0(y) := 0$ for $y>x$ and compute
	\begin{align*}
		\|u_x\|_{L^2([0,x])}^2 = \|u_x^0\|_{L^2([0,\pi])}^2 
		&\leq C_M^{(2)} \sum_{n=0}^\infty \big|\langle u_x^0,g_n\rangle_{L^2([0,\pi])}\big|^2
		\\
		&= C_M^{(2)} \sum_{n=0}^\infty \big|\langle u_x,g_n\rangle_{L^2([0,x])}\big|^2
		\\
		&= C_M^{(2)}\langle f_x, u_x\rangle_{L^2([0,x])}
		\\
		&\leq C_M^{(2)}\|f_x\|_{L^2([0,x])} \|u_x\|_{L^2([0,x])},
	\end{align*}
which implies
	\begin{equation*}
		 \|u_x\|_{L^2([0,x])} \leq C_M^{(2)}\|f_x\|_{L^2([0,x])}.
	\end{equation*}
	
	To prove (ii), we use the regularizing properties of $\kappa_x$, namely if $f_x\in L^\infty([0,x])$, then by \eqref{eq:integral_equation2_new} and H\"older's inequality we have
	\begin{align*}
		|u_x(y)| &\leq |f_x(y)| + \left|\int_0^x k(y,s)u_x(s)\,ds\right|
		\\
		&\leq \|f_x\|_{L^\infty} + \pi^{\f12}\|k\|_\infty\|u_x\|_{L^2}
		\\
		&\leq \|f_x\|_{L^\infty} + \pi^{\f12}\|k\|_\infty C_M^{(2)}\|f_x\|_{L^2}
	\end{align*}
which implies
	\begin{equation*}
	 \|u_x\|_{L^\infty} \leq \left(\pi^{\f12}+\pi\|k\|_\infty C_M^{(2)}\right)\|f_x\|_{L^\infty}.
	\end{equation*}
	This completes the proof.
\end{proof}

\medskip

We can now finalize the proof of Theorem \ref{thm:main2}\eqref{eq:mainth3}. Recall that our starting point was eq. \eqref{eq:explicit_bounds0},
\begin{align*}
	\|u_{x,N}-u_x\|_{L^\infty([0,x])} &\leq \|(I-\kappa_x)^{-1}\|_{L^\infty\to L^\infty}
	\\
	\times& \Big( \|f_{x,N} - f_x\|_{L^\infty([0,x])}  + \pi \|k_N-k\|_{L^\infty([0,\pi]^2)} \|u_{x,N}\|_{L^2([0,x])} \Big),
\end{align*}
where we wanted to bound each  term. First, we note that from \eqref{eq:Phi_explicit_bound} follows the uniform bound $\|k\|_{L^\infty} \leq C_M^{(1)}$.  Combining this bound with Lemmas \ref{lemma:explicit_bounds_F-FN} and \ref{lemma:explicit_bound_(I-k)} we have
\begin{align}
	\|(I-\kappa_x)^{-1}\|_{L^2\to L^2} &\leq C_M^{(2)},
	\label{eq:ex_bounds0}
	\\[2mm]
	\|(I-\kappa_x)^{-1}\|_{L^\infty\to L^\infty} &\leq \pi^{\f12}+\pi C_M^{(1)} C_M^{(2)},
	\label{eq:ex_bounds1}
	\\[2mm]
	\sup_{x\in[0,\pi]}\|f_{x,N} - f_x\|_{L^\infty([0,x])} &\leq C_M^{(1)} N^{-\f12},
	\label{eq:ex_bounds2}
	\\
	\|k_N-k\|_{L^\infty([0,\pi]^2)} &\leq C_M^{(1)} N^{-\f12}.
	\label{eq:ex_bounds3}
\end{align}
Moreover, using Lemma \ref{lemma:T_inverse_bound} and \eqref{eq:Phi_explicit_bound} again, we have
\begin{align}
\begin{aligned}
	\|u_{x,N}\|_{L^2([0,x])} &\leq \|(I-\kappa_{x,N})^{-1}\|_{L^2([0,x])\to L^2([0,x])}\|f_{x,N}\|_{L^2}
	\\
	&\leq \f{\|(I-\kappa_x)^{-1}\|_{L^2\to L^2}}{1-\|\kappa_x-\kappa_{x,N}\|_{L^2\to L^2}\|(I-\kappa_x)^{-1}\|_{L^2\to L^2}} \pi^{\f12}C_M^{(1)},
	\label{eq:u_N_explicit_bound}
\end{aligned}
\end{align}
where the last line holds if $\|\kappa_x-\kappa_{x,N}\|_{L^2\to L^2}<\|(I-\kappa_x)^{-1}\|_{L^2\to L^2}^{-1}$. But this last condition can be ensured  by choosing $N$ large enough: by \eqref{eq:ex_bounds1}-\eqref{eq:ex_bounds3} the choice 
$ N > \pi (C_M^{(1)}C_M^{(2)})^2$ is sufficient, where we  remind that this constant is explicitly computable. For the sake of definiteness, let us assume from now on that $N>N_0$, where
\begin{align*}
	N_0 := 2 \pi \big(C_M^{(1)}C_M^{(2)}\big)^2.
\end{align*}
This choice ensures $\|\kappa_x-\kappa_{x,N}\|_{L^2\to L^2}\|(I-\kappa_x)^{-1}\|_{L^2\to L^2}< 1/2$.
Inserting the necessary bounds into \eqref{eq:u_N_explicit_bound} we obtain
\begin{align}\label{eq:ex_bounds4}
	\|u_{x,N}\|_{L^2([0,x])} &< 2\pi^{\f12}C_M^{(1)}C_M^{(2)} \quad \text{ for }N>N_0.
\end{align}
Using \eqref{eq:ex_bounds1}-\eqref{eq:ex_bounds3} and \eqref{eq:ex_bounds4} in \eqref{eq:explicit_bounds0} we have that for all  $N>N_0$ 
\begin{align*}
	\|u_{x,N}-u_x\|_{L^\infty([0,x])} &\leq 
	\Big(\pi^{\f12}+\pi C_M^{(1)} C_M^{(2)}\Big)\Big( C_M^{(1)}N^{-\f12} + \pi C_M^{(1)}N^{-\f12} 2\pi^{\f12}C_M^{(1)}C_M^{(2)} \Big)
	\\
	&= C_M^{(1)}N^{-\f12}\Big(\pi^{\f12}+\pi C_M^{(1)} C_M^{(2)}\Big)\Big( 1 + 2\pi^{\f32} C_M^{(1)}C_M^{(2)} \Big).
\end{align*}
Taking the supremum over $x\in[0,\pi]$ and reverting back to the classical notation, we have shown 
\begin{align}\label{eq:K-KN_explicit_bound}
	\|K_N-K\|_{L^\infty([0,\pi]^2)} &\leq N^{-\f12} C_M^{(1)}\Big(\pi^{\f12}+\pi C_M^{(1)} C_M^{(2)}\Big)\Big( 1 + 2\pi^{\f32} C_M^{(1)}C_M^{(2)} \Big).
\end{align}
Now the calculation \eqref{eq:q-error-calculation} implies
\begin{align}\label{eq:q_final_bound}
	\|q_N-q\|_{W^{-1,\infty}([0,\pi])} &\leq 2N^{-\f12} C_M^{(1)}\Big(\pi^{\f12}+\pi C_M^{(1)} C_M^{(2)}\Big)\Big( 1 + 2\pi^{\f32} C_M^{(1)}C_M^{(2)} \Big),
\end{align}
It remains to prove error bounds for the boundary conditions $h$, $H$. To this end, note that $h = K(0,0)$. Thus, with $h_N := K_N(0,0)$, eq. \eqref{eq:K-KN_explicit_bound} implies 
\begin{align}\label{eq:h_final_bound}
	|h-h_N| &\leq N^{-\f12} C_M^{(1)}\Big(\pi^{\f12}+\pi C_M^{(1)} C_M^{(2)}\Big)\Big( 1 + 2\pi^{\f32} C_M^{(1)}C_M^{(2)} \Big).
\end{align}
To compute $H$, recall our assumption $\omega=0$, hence by \eqref{eq:asymptotics} we have
\begin{align*}
	H &= - h - \f12\int_0^\pi q(x)\,dx
	\\
	&= - h - K(\pi,\pi).
\end{align*}
Hence we may define the computable approximation $H_N:=-h_N-K_N(\pi,\pi)$. Then by \eqref{eq:K-KN_explicit_bound} we have
\begin{align}\label{eq:H_final_bound}
	|H-H_N| &\leq |h-h_N| + |K(\pi,\pi) - K_N(\pi,\pi)|
	\nonumber
	\\[1mm]
	&\leq 2N^{-\f12} C_M^{(1)}\Big(\pi^{\f12}+\pi C_M^{(1)} C_M^{(2)}\Big)\Big( 1 + 2\pi^{\f32} C_M^{(1)}C_M^{(2)} \Big).
\end{align}
Together, \eqref{eq:q_final_bound}, \eqref{eq:h_final_bound} and \eqref{eq:H_final_bound} imply that the algorithm
\begin{align*}
	\Gamma_N(\lambda,\alpha) := (q_N,h_N,H_N)
\end{align*}
achieves explicit error control with convergence rate $N^{-\f12}$ and computable constants $C_M^{(1)}$, $C_M^{(2)}$. We immediately conclude that $(\Omega_{0,M},\Lambda,\mathcal{M}_\infty,\Xi_\infty) \in \Delta_1^A$. Showing that $(\Omega_{0,M},\Lambda,\mathcal{M}_p,\Xi_p) \in \Delta_1^A$ for $p<+\infty$ follows immediately, using H\"older's inequality.
\qed
\section{Numerical Results}\label{sec:numerics}
The algorithm that computes $q_N$ by solving \eqref{eq:linear_system} can straightforwardly be implemented in MATLAB. Figure \ref{fig:RS_potentials} shows the reconstruction of 3 different potentials from 10 and 30 eigenvalues and norming constants, respectively, with Neumann boundary conditions $h=H=0$. These potentials were previously suggested in \cite{Rundell-Saks} to test the performance of reconstruction algorithms in different situations. The first potential, $q^{(1)}$, is a smooth function, the second, $q^{(2)}$, is discontinuous and the third, $q^{(3)}$, is continuous with discontinuous derivative.

The spectra and norming constants were computed in MATLAB using the MATSLISE package \cite{matslise}. To improve convergence, the reconstruction algorithm approximates $\omega/\pi$ by $\varpi:=\lambda_{N-1}-(N-1)^2$ (cf. \eqref{eq:asymptotics}) and then applies the algorithm from Section \ref{sec:finite} to the set $\{\lambda_n-\varpi\}_{n=0}^{N-1}$. For reasons of practicality, we do not compute the derivative of $K(\cdot,\cdot)$ symbolically, as suggested below eq. \eqref{eq:u_formula}. Rather, we compute the values $K(x,x)$ exactly and then differentiate numerically.

Re-computing the spectrum for the reconstructed potential $q_{N-1}$ (and boundary conditions $h_{N-1}$, $H_{N-1}$)  with MATSLISE gives good agreement with the original spectrum. Denoting the original eigenvalues by $\{\lambda_0,\dots,\lambda_{N-1}\}$ and the ones obtained from $q_{N-1}$ by $\{\mu_0,\dots,\mu_{N-1}\}$ and measuring the error by 
\begin{align}\label{eq:e_N_def}
	e_N &:= \sqrt{ \sum_{i=0}^{N-1} \big(\lambda_i^{\f12} - \mu_i^{\f12}\big)^2 }
%	\\
%	&= \|\lambda^{\f12} - \mu^{\f12}\|_{\ell^2}
\end{align}
gives us a means of assessing performance.
In each case the reconstruction gives an error of at most $e_N\lesssim 8\cdot10^{-8}$. The precise values for $e_N,h_N,H_N$ are shown in Table \ref{tab:ehH}.
\renewcommand{\arraystretch}{1.2}
\begin{table}[H]
\begin{tabular}{l c c c}
	& & \boldmath${N=9}$: & \\
	\toprule
	 {Potential} & $h_N$ & $H_N$ & $e_N$\\
	\hline
	$q^{(1)}$ & $-0.015$ & $0.016$ & $1.1\cdot 10^{-9}$\\
	$q^{(2)}$ & $-0.041$ & $0.021$ & $4.5\cdot 10^{-8}$\\
	$q^{(3)}$ & $-0.031$ & $0.011$ & $3.1\cdot 10^{-9}$\\
	\bottomrule
\end{tabular}
\qquad
\begin{tabular}{l c c}
	& \boldmath${N=29}$: & \\
	\toprule
	  $h_N$ & $H_N$ & $e_N$\\
	\hline
	 $-0.005$ & $0.005$ & $5\cdot 10^{-11}$\\
	 $-0.014$ & $-0.011$ & $5\cdot 10^{-8}$\\
	 $-0.010$ & $0.008$ & $7.6\cdot 10^{-8}$\\
	\bottomrule
\end{tabular}
\caption{Computed values for boundary conditions and spectral error for the potentials in Figure \ref{fig:RS_potentials}.}
\label{tab:ehH}
\end{table}
Note that the more eigenvalues are in play, the longer the sum in \eqref{eq:e_N_def} becomes. 

The Matlab code that produced Figure \ref{fig:RS_potentials} and the values in Table \ref{tab:ehH} is openly available at \url{https://github.com/frank-roesler/inverse_SCI}.

\begin{figure}[htbp]
	\centering
	\includegraphics{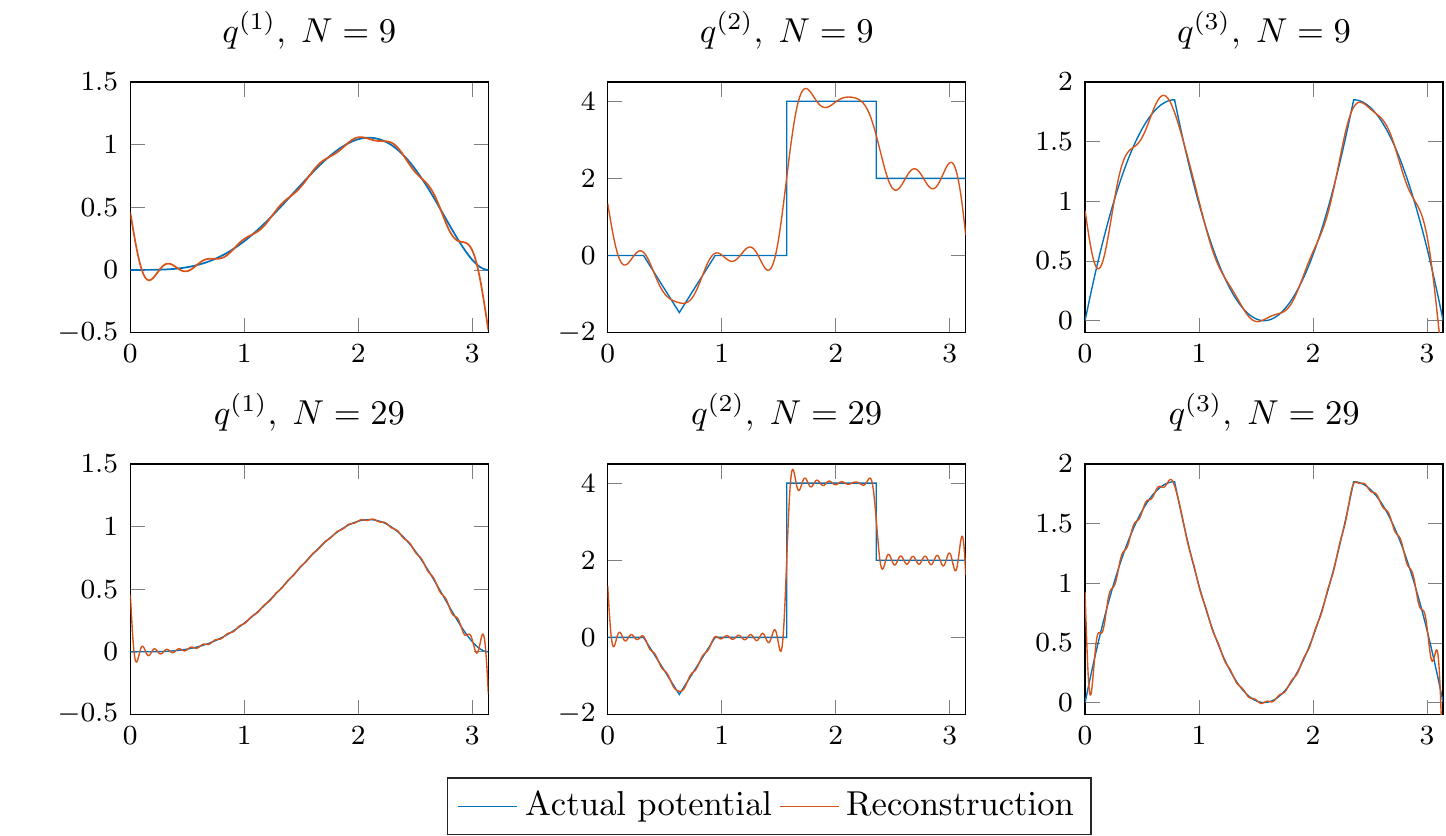}
	\caption{Reconstruction of different potentials by solving \eqref{eq:linear_system}. Top row: reconstruction from 10 eigenvalues and norming constants; bottom row: reconstruction from 30 eigenvalues and norming constants.}
	\label{fig:RS_potentials}
\end{figure}

\appendix

\section{Further Implications}
In this section we provide further implications of Theorem \ref{thm:main2}. According to Theorem \ref{thm:main2}\eqref{eq:mainth1} the potential can be reconstructed exactly from a finite number $N$ of `nontrivial' eigenvalues and norming constants (by `trivial' we mean $\lambda_n=n^2$ and $\alpha_n=\frac\pi2$). What if the number $N$ of spectral data is not known \emph{a priori}? The result below shows that we can retain the $\Delta_0^A$ result if we replace knowledge of the number $N$, with knowledge that:
	\begin{quote}
	If, for some given $\tilde{n}$, there are $\tilde{n}$ consecutive $(\lambda_n,\alpha_n)$ that are `trivial' then all subsequent $(\lambda_n,\alpha_n)$  are `trivial'.
	\end{quote}
This provides us with a mechanism to stop looking for additional spectral data after a finite amount of time.
\begin{corollary}
	Theorem \ref{thm:main2} immediately implies the following classification.
	Let $\tilde n\in\N$. Denote by $\tilde\Omega_{\tilde n}\subset\Omega$ the set of $(\lambda,\alpha)$ such that there exists $n_0\in\N$ such that if $\lambda_n = n^2$ and $\alpha_n = \f{\pi}{2}$ for all $n\in\{n_0+1,\dots,n_0+\tilde n\}$, then $\lambda_n = n^2$ and $\alpha_n = \f{\pi}{2}$ for all $n> n_0$.
		Then for all $\tilde n\in\N$ one has
		\begin{align*}
			(\tilde\Omega_{\tilde n},\Lambda,\mathcal{M}_\mathrm{disc},\Xi_\mathrm{disc}) &\in \Delta_0^A.
		\end{align*}
\end{corollary}

\begin{proof}
	Let $\tilde n\in\N$. It is easy to see that every $(\lambda,\alpha)\in\tilde\Omega_{\tilde n}$ is in $\Omega_N$ for some $N=N(\lambda,\alpha)\in\N$. In fact, this number $N(\lambda,\alpha)$ can be computed in finite time, as Algorithm \ref{alg:compute_N} below shows. 
	Therefore we may define 
	\begin{align*}
		\Gamma(\lambda,\alpha) := \Gamma_{N(\lambda,\alpha)}^{\text{fin}}(\lambda,\alpha)
	\end{align*}
	where $N(\lambda,\alpha)$ is computed by Algorithm \ref{alg:compute_N} and $ \Gamma_{N}^{\text{fin}}$ is the algorithm provided by Theorem \ref{thm:main2}\eqref{eq:mainth1}. Since all computations terminate in finite time, it follows immediately that $\tilde\Omega_{\tilde n}\in \Delta_0^A$.
	\medskip

\begin{algorithm}[H]
	\setlength{\lineskip}{1mm}
%	\DontPrintSemicolon
	\SetAlgorithmName{Algorithm}{}{}
	\SetAlgoInsideSkip{smallskip}
	\SetAlgoSkip{bigskip}
	\SetAlgoLined
	\KwIn{$\tilde n$ and $(\lambda,\alpha)\in\tilde\Omega_{\tilde n}$}
	Initialize $\mathrm{ctr}=0$, $n=1$\;
	 \While{True}{
	 	\eIf{$\lambda_n = n^2$ \textbf{\emph{and}} $\alpha_n = \nicefrac\pi2$}{
	 		$\mathrm{ctr} := \mathrm{ctr}+1$\;
	 	}{
	 		$\mathrm{ctr}:=0$\;
	 	}
	 	\eIf{$\mathrm{ctr} = \tilde n$}{
	 		break out of loop\;
	 	}{
 		 	$n := n+1$\;
 		 }
	 }
	 \Return{$N:=n-\mathrm{ctr}+1$}
	 \caption{Compute $N(\lambda,\alpha)$}
	 \label{alg:compute_N}
\end{algorithm}
\end{proof}

%\nocite{*}
\bibliography{mybib}
\bibliographystyle{abbrv}

\end{document}

%% file: preamble.tex
\newcommand{\f}{\frac}

\newcommand{\del}{\partial}

\newcommand{\R}{\mathbb R}
\newcommand{\C}{\mathbb C}
\newcommand{\N}{\mathbb N}

\newcommand{\eps}{\varepsilon}
\renewcommand{\epsilon}{\varepsilon}

\newcommand{\Om}{\Omega}

\newcommand{\h}{\mathcal{H}}

\newenvironment{enumi}{\begin{enumerate}[(i)]}{\end{enumerate}}

\AtBeginDocument{

}

\usepackage{etex}
\usepackage[english]{babel}
\usepackage{tikz}
\usetikzlibrary{arrows,decorations.pathmorphing,backgrounds,positioning,fit,petri,patterns}
\usepackage{pgfplots}
\usepackage[utf8]{inputenc}
\usepackage{graphicx}
\usepackage[colorlinks=true,linkcolor=red,citecolor=blue]{hyperref}
\usepackage{amsfonts}
\usepackage{enumerate}
\usepackage{booktabs}
\usepackage{array} 
\usepackage{paralist} 
\usepackage{subfig} 
\usepackage{mathtools}
\usepackage{tabu}
\usepackage{amsthm}
\usepackage{empheq}
\usepackage{amsopn}
\usepackage{dsfont}
\usepackage{wrapfig}
\usepackage[font=small]{caption}
\usepackage{units}
\usepackage[symbol]{footmisc}
\usepackage{todonotes}
\usepackage{amssymb}
\usepackage{pdfpages}
\usepackage{setspace}
\usepackage{float}
\usepackage{scrextend}

\makeatletter
\renewenvironment{cases}[1][l]{\matrix@check\cases\env@cases{#1}}{\endarray\right.}
\def\env@cases#1{%
  \let\@ifnextchar\new@ifnextchar
  \left\lbrace\def\arraystretch{1.2}%
  \array{@{}#1@{\quad}l@{}}}
\makeatother

\usepackage{kvsetkeys}% provides \comma@parse
\usepackage{etexcmds}% provides \etex@unexpanded

\makeatletter
\newcount\arg@count
\newcommand{\arg@parser}[1]{%
  \advance\arg@count\@ne
  \expandafter\let\csname arg\romannumeral\arg@count\endcsname\comma@entry
}
\newcommand\res[1]{% Default is empty and will be configured later
  \arg@count=\z@
  \comma@parse{ \lambda,A }\arg@parser % Default values
  \arg@count=\z@
  \comma@parse{#1}\arg@parser
  \ifnum\arg@count>2 %
    \@latex@error{Too many arguments}{%
      The macro \string\mycmd\space got \the\arg@count\space
       arguments,\MessageBreak
      but expected are 2 arguments.\MessageBreak
      \@ehd
    }%
  \fi
  \edef\process@me{%
    \noexpand\@res
    {\etex@unexpanded\expandafter{\argi}}%
    {\etex@unexpanded\expandafter{\argii}}%
  }%
  \process@me
}
\newcommand{\@res}[2]{%
  \ensuremath\left( #1 - #2 \right)^{-1}
}
\makeatother

\makeatletter
\newcount\arg@count
\newcommand\p[1]{% Default is empty and will be configured later
  \arg@count=\z@
  \comma@parse{  }\arg@parser % Default values
  \arg@count=\z@
  \comma@parse{#1}\arg@parser
  \ifnum\arg@count=2 %
  \else
    \@latex@error{Wrong number of mandatory arguments}{%
      The macro \string\p\space got \the\numexpr\arg@count-2\relax\space
      mandatory arguments,\MessageBreak
      but expected are 3 mandatory arguments.\MessageBreak
      \@ehd
    }%
  \fi
  \edef\process@me{%
    \noexpand\@p
    {\etex@unexpanded\expandafter{\argi}}%
    {\etex@unexpanded\expandafter{\argii}}%
  }%
  \process@me
}
\newcommand{\@p}[2]{%
  \ensuremath \left\langle #1 , #2 \right\rangle
}
\makeatother

\allowdisplaybreaks
\numberwithin{equation}{section}

\theoremstyle{definition}
\newtheorem{de}{Definition}[section]
\newtheorem{hyp}[de]{Hypothesis}
\theoremstyle{plain}
\newtheorem{prop}[de]{Proposition}
\newtheorem{lemma}[de]{Lemma}
\newtheorem{theorem}[de]{Theorem}
\newtheorem{notation}[de]{Notation}

\newtheorem{corollary}[de]{Corollary}

\theoremstyle{remark}
\newtheorem{remark}[de]{Remark}